\documentclass[reqno,11pt]{amsart}
\usepackage{amsmath, amsfonts,amsthm,amssymb,amsbsy,upref,color,graphicx,amscd,hyperref,enumerate}
\usepackage[active]{srcltx}
\usepackage[latin1]{inputenc}
\usepackage{bbm}
\usepackage{pgf}
\usepackage{xfrac}

\def\ds{\displaystyle}
\def\eps{{\varepsilon}}
\def\N{\mathbb{N}}

\def\R{\mathbb{R}}

\def\HH{\mathcal{H}}

\newcommand{\be}{\begin{equation}}
\newcommand{\ee}{\end{equation}}

\newcommand{\de}{\partial}
\newcommand{\dist}{{\rm {dist}}}

\theoremstyle{plain}
\newtheorem{theo}{Theorem}

\numberwithin{equation}{section}
\theoremstyle{plain}
\newtheorem{teo}{Theorem}[section]
\newtheorem{lemma}[teo]{Lemma}
\newtheorem{cor}[teo]{Corollary}
\newtheorem{prop}[teo]{Proposition}

\newtheorem{oss}[teo]{Remark}

\title[On the logarithmic epiperimetric inequality]{On the logarithmic epiperimetric inequality for the obstacle problem}

\author[L.~Spolaor]{Luca Spolaor}
\address{\textit{L.~Spolaor:} Department of Mathematics, University of California, San Diego, La Jolla, CA, 92093
}
\email{lspolaor@ucsd.edu}

\author[B.~Velichkov]{Bozhidar Velichkov}
\address{\textit{B.~Velichkov:}
	Dipartimento di Matematica,
	Universit\`a di Pisa\\
	Largo B. Pontecorvo 5,
	56127 Pisa}
\email{bozhidar.velichkov@unipi.it}

\makeindex

\begin{document}


\begin{abstract}
We give three different proofs of the log-epiperimetric inequality at singular points for the obstacle problem. In the first, direct proof, we write the competitor explicitly; the second proof is also constructive, but this time the competitor is given through the solution of an evolution problem on the sphere. We compare the competitors obtained in the different proofs and their relation to other similar results that appeared recently. Finally, in the appendix, we give a general theorem, which can be applied also in other contexts and in which the construction of the competitor is reduced to finding a flow satisfying two differential inequalities. 
%
\end{abstract}

\maketitle
\begin{center}
	\it To Sandro Salsa on the occasion of his 70th birthday\rm 	
\end{center}




\section{Introduction}	
%


For any $u\in H^1(B_1)$, we define the functionals (\emph{Weiss' boundary adjusted energies})
\begin{equation*}
W_0(u):=\int_{B_1}|\nabla u|^2\,dx-2\int_{\partial B_1}u^2\,d\HH^{d-1}\qquad\text{and}\qquad W(u):=W_0(u)+\int_{B_1}u\,dx.
\end{equation*}
We denote by $\mathcal S$ the following class of $2$-homogeneous polynomials :
\begin{equation}
\label{e:critical_points}
\begin{array}{ll}
\ds\mathcal S:=\Big\{Q_A \colon \R^d \to \R &:\, Q_A(x) = x \cdot Ax,\text{ where $A=(a_{ij})_{ij}$  is a symmetric}\\
&\qquad\qquad \ds\text{ non-negative matrix such that}\quad {\rm{tr}} A =\sum_{i=1}^d a_{ii}= \frac14\ \Big\}\,.
\end{array} 
\end{equation}
Notice that the functional $W$ is constant on $\mathcal S$. We will use the notation 
\begin{equation}\label{e:critical_points_energy}
W(\mathcal S):=W(Q_A)\quad \mbox{for any}\quad Q_A\in \mathcal S.
\end{equation}
This paper is dedicated to the so-called {\it logarithmic epiperimetric inequality}, which was first introduced in \cite{cospve1}, for the functional $W$ and the set $\mathcal S$, and already found several applications to different variational free boundary problems (see for instance \cite{cospve2}, \cite{shi}, \cite{esv1}, \cite{esv2}). 
\begin{theo}[Log-epiperimetric inequality for $W$]\label{t:epi:sing}
There are dimensional constants $\delta>0$, $\eps>0$ and $\gamma\in[0,1)$ such that the following claim holds. For every non-negative function $c\in H^1(\partial B_1)$, with $2$-homogeneous extension $z$ on $B_1$, satisfying   
	$$\text{\rm dist}_{2}\left(c,\mathcal S\right)\le \delta\qquad\text{and}\qquad  W(z)-W(\mathcal S)\le 1,$$
{\flushright there is a non-negative function $h\in H^1(B_1)$ with $h=c$ on $\partial B_1$ satisfying the inequality }
\begin{equation}\label{e:sing:epi}
W(h)-W(\mathcal S) \le \big(W(z)-W(\mathcal S)\big)\Big(1-\eps\big|W(z)-W(\mathcal S)\big|^{\gamma} \Big).
\end{equation}
\end{theo}

\subsection{Epiperimetric and log-epiperimetric inequalities - overview of the known results} The epiperimetric inequalities are powerful tools in the regularity theory of free boundary problems and minimal surfaces. The concept of an epiperimetric inequality (which is \eqref{e:sing:epi} with $\gamma=0$) was first introduced by Reifenberg in \cite{reifenberg} in the context of minimal surfaces in the 60s. In the late 90s, Weiss \cite{weiss} used an epiperimetric inequality approach to study the free boundaries of the obstacle problem. 

The Weiss' epiperimetric inequality was still of the form \eqref{e:sing:epi}, with $\gamma=0$ and for the same energy $W$, but the set $\mathcal S$ defined in \eqref{e:critical_points} was replaced by the set of the 'flat' blow-up limits 
$$\mathcal R=\Big\{Q(x)=\frac14(x\cdot\nu)_+^2\quad\text{where}\quad \nu\in\partial B_1\Big\}.$$
In \cite{weiss} Weiss used the epiperimetric inequality to prove the $C^{1,\alpha}$ regularity of the 'flat' free boundaries. 
It is now known that the epiperimetric inequality (\eqref{e:sing:epi} with $\gamma=0$) cannot hold in a neighborhood of the set $\mathcal S$; this follows from the counterexample to the $C^{1, \alpha}$ regularity of the singular free boundaries given in \cite{FiSe}.

There are several other epiperimetric inequalities for free boundary problems in the literature. By using the technique of Weiss, Focardi-Spadaro \cite{fosp} and Garofalo-Petrosyan-Smit\, Vega\,Garcia \cite{gapeve} proved an epiperimetric inequality at the regular points of the thin-obstacle free boundaries. Then, in \cite{spve}, by using a different (direct) approach, we proved an epiperimetric inequality for the Bernoulli free boundary problems in dimension two. 
\medskip

The log-epiperimetric inequality (\eqref{e:sing:epi} with $\gamma\in[0,1)$) was introduced
 in our work \cite{cospve1}, in collaboration with M.\,Colombo, where we first proved Theorem \ref{t:epi:sing}. The initial idea in \cite{cospve1} was to attack the epiperimetric inequality for $\mathcal S$ (left open in the work of Weiss \cite{weiss}) by the direct approach from \cite{spve}. As we already mentioned above, this cannot be actually done, but it led to the formulation of the log-epiperimetric inequality \eqref{e:sing:epi}, from which we obtained the $C^{1,\log}$ regularity of the singular part of the free boundary (from where the name {\it logarithmic}).

 Following the original spirit of \cite{spve}, our approach in \cite{cospve1} was still {\it direct} in the sense that we built the competitor explicitly. We later used this idea to prove a log-epiperimetric inequality at the singular points of the thin-obstacle free boundaries \cite{cospve2}. Constructing explicit competitors has the advantage to provide proofs that use only elementary tools and essentially boils down to constructing sub and supersolutions starting from a trace, which is close to the set $\mathcal S$. This requires the set $\mathcal S$ to be known explicitly. On the other hand, one is mainly interested in the case when $\mathcal S$ is a set of global homogeneous solutions to some free boundary variational problem and the classifications of these solutions is known only for some specific problems.
 
 In \cite{esv1} and \cite{esv2}, in collaboration with Max Engelstein, we elaborated a different approach to the log-epiperimetric inequality and we constructed the competitor by reparametrizing the solution of an evolution problem on the sphere.  Finally, in \cite{cospve3}, we exploited some of the ideas from \cite{esv1} and \cite{esv2} to give new proofs of the log-epiperimetric inequalities (for the obstacle and the thin-obstacle problems) from \cite{cospve1} and \cite{cospve2}, and we showed its relation to a class of parabolic variational inequalities.

\subsection{Aim and organization of the paper} This paper has several objectives. Our first aim is to provide a different (and hopefully easier) direct proof of Theorem \ref{t:epi:sing}. Our new proof (see Section \ref{s:direct}) is based on a specific decomposition of the trace inspired by the competitor that we used to prove the constrained \L ojasiewicz inequality in \cite{cospve3}. Then, we notice that the competitor built this way in fact simulates the behavior the gradient flow from \cite{cospve3} and so, we give a second proof of Theorem \ref{t:epi:sing} (see Section \ref{s:constructive}) by using a construction in the spirit of \cite{cospve3}. 

We explain the ideas behind the two different constructions in Section \ref{s:survey}, which can be read independently. Moreover, in order to give our second proof of Theorem \ref{t:epi:sing}, we give a new general result in which the construction of the competitor is reduced to finding a flow on the sphere that satisfies two specific inequalities (see Theorem \ref{p:epik}). This general result applies both to the gradient flow from \cite{cospve3} and to the flow from Section \ref{s:constructive}. It is also intended to facilitate the construction of competitors for other variational problems, for instance, non-local obstacle problems.

\subsubsection{A direct proof of Theorem \ref{t:epi:sing}} In the first three sections we give a direct proof of Theorem \ref{t:epi:sing} in the spirit of \cite{cospve1}, but using a different competitor. Section \ref{s:prelimaries} contains the notations and some basic facts about the functional $W$. Section \ref{s:best} is dedicated to the main estimate that we need in the proof of Theorem \ref{t:epi:sing}. In Section \ref{s:direct} we give the proof of the theorem by putting together the estimates from Section \ref{s:prelimaries} and Section \ref{s:best}. 

\subsubsection{The constructive approach to Theorem \ref{t:epi:sing}} In Section \ref{s:constructive} and Section \ref{s:third_proof} we give a proof of Theorem \ref{t:epi:sing} based on the definition of a flow on the sphere. Then, we use the estimates from Section \ref{s:prelimaries} and Section \ref{s:best} to prove that these flows satisfy the condition of Theorem \ref{p:epik}. 

\subsubsection{Comparison of the direct and the constructive approach} In Section \ref{s:survey}, we explain the main ideas behind the direct proofs from Section \ref{s:direct} and \cite{cospve1}, and the constructive, gradient flow approach from Sections \ref{s:constructive} and \ref{s:third_proof}. This section can be read independently; we only use some of the notations and preliminary results from Section \ref{s:prelimaries}. 

\subsubsection{The general Theorem \ref{p:epik}} In Theorem \ref{p:epik} we show how to construct a competitor out of a flow defined in $H^1(\partial B_1)$; we prove that competitor satisfies a log-epiperimetric inequality provided that the flow satisfies two main conditions: 
\begin{center}
\it an energy dissipation inequality \eqref{e:energy_dissipation_condition}, \rm \quad and \quad \it \L ojasiewicz inequality \eqref{e:lojaK}.
\end{center}

\noindent This result applies to both the flow from Section \ref{s:constructive} and the gradient flow from Section \ref{s:third_proof}; in the first case, the exponent $\gamma$ in the log-epiperimetric inequality \eqref{e:sing:epi} appears as a consequence of the energy dissipation inequality \eqref{e:energy_dissipation_condition}, while in the case of the gradient flow (Section \ref{s:third_proof}), the exponent is due to the \L ojasiewicz inequality  \eqref{e:lojaK}. In both cases, Theorem \ref{p:epik} allows to reduce the proof of the log-epiperimetric inequality to verifying that \eqref{e:energy_dissipation_condition} and \eqref{e:lojaK} hold along the flow.

\subsection{Log-epiperimetric inequality and structure of the singular free boundaries}$ $\\
 A consequence of Theorem \ref{t:epi:sing} is the following result on the structure of the singular free boundaries of solutions to the obstacle problem, which we give here for the sake of completeness. We recall that $u:B_1\to\R$ is a solution to the obstacle problem (in the unit ball $B_1\subset\R^d$) if $u\ge 0$ and 
\begin{align}
\int_{B_1}\big(|\nabla u|^2+u\big)\,dx\le \int_{B_1}\big(|\nabla v|^2&+v\big)\,dx\quad\text{for every}\quad v\in H^1(B_1)\quad\notag \\
&\text{such that}\quad v\ge 0\quad\text{in}\quad B_1\quad\text{and}\quad u-v\in H^1_0(B_1).\label{e:opb}
\end{align}
We define $\Omega_u:=\{u>0\}$ and the set of singular points
$$\text{\rm Sing}(\partial\Omega_u):=\Big\{x_0\in\partial\Omega_u\ :\ \lim_{r\to0}\frac{\big|B_r(x_0)\cap\Omega_u\big|}{|B_r|}=1\Big\}.$$
Let $u$ be a solution to the obstacle problem. We say that $Q:B_1\to\R$ is a blow-up limit of $u$ at $x_0\in\partial\Omega_u\cap B_1$ (and we write $Q\in BU(u,x_0)$), if there is a sequence $r_n\to 0$ such that 
$$\lim_{n\to\infty}\|u_{r_n,x_0}-Q\|_{L^2(\partial B_1)}\qquad\text{where}\qquad u_{r_n,x_0}(x)=\frac1{r_n^2}u(x_0+r_nx)\,.$$
It is well-known that $x_0\in \text{\rm Sing}(\partial\Omega_u)$ if and only if $BU(u,x_0)\subset \mathcal S$ (see, for instance, \cite{caffarelli_revisited} and \cite{figalli}). 

\noindent Finally, we define the strata $\text{\rm Sing}_k(\partial\Omega_u)$, for every $k\in\{0,1,\dots,d-1\}$, as 
$$\text{\rm Sing}_k(\partial\Omega_u):=\Big\{x_0\in \text{\rm Sing}(\partial\Omega_u)\ :\ \text{there is }\ Q_A\in BU(u,x_0)\ \text{ such that }\ \text{dim}\big(\text{Ker}\,A\big)=k\Big\}.$$
As an immediate consequence of Theorem \ref{t:epi:sing}, we obtain Theorem \ref{t:reg}, proved in \cite{cospve1}. A finer result on the structure of the singular set (in any dimension) was obtained by Figalli and Serra \cite{FiSe}, and more recently, by Figalli, Serra and Ros-Oton \cite{fisero}.
\begin{theo}[Structure of the singular free boundaries]\label{t:reg}
Let $u:B_1\to\R$ be a solution to the obstacle problem \eqref{e:opb}. 
Then, the following holds. 
\begin{enumerate}[(i)]
\item {\bf Uniqueness of the blow-up limit.} At every $x_0\in \text{\rm Sing}(\partial\Omega_u)$, the blow-up limit is unique, that is, there is $Q_{A_{x_0}}\in\mathcal S$ such that 
$$\lim_{r\to0}\|u_{r,x_0}-Q_{A_{x_0}}\|_{L^2(\partial B_1)}=0.$$
\item {\bf Rate of convergence of the blow-up sequence.} For every $x\in B_1$, there is a ball $B=B_\rho(x)$ and $R>0$ such that 
$$\|u_{r,x_0}-Q_{A_{x_0}}\|_{L^2(\partial B_1)}\le C(-\ln r)^{-\delta}\quad\text{for every}\quad x_0\in  \text{\rm Sing}(\partial\Omega_u)\cap B\quad\text{and every}\quad r\le R.$$
\item {\bf Distance between the blow-up limits at different points.} For every $x\in B_1$, there is a ball $B_\rho(x)$ such that 
$$\|Q_{A_{y_0}}-Q_{A_{x_0}}\|_{L^2(\partial B_1)}\le C\big(-\ln |x_0-y_0|\big)^{-\delta}\quad\text{for every}\quad x_0,y_0\in  \text{\rm Sing}(\partial\Omega_u)\cap B_\rho(x).$$
\item {\bf Structure of the strata.} The set $\text{Sing}_{\,0}(\partial\Omega_u)$ is discrete. For every  $1\le k\le d-1$, there is $\delta$, depending on $d$ and eventually on $k$ such that the following holds. 

For every $x\in \text{Sing}_{\,k}(\partial\Omega_u)$, there is a ball $B_\rho(x)\subset B_1$ such that $\text{Sing}_{\,k}(\partial\Omega_u)\cap B_\rho(x)$ is contained in a $C^{1,\log}$-regular $k$-dimensional manifold.  Precisely, up to a rotation of the coordinate system, there is a $C^1$ regular function $\varphi:\R^k\to\R^{d-k}$ such that 
$$\text{Sing}_{\,k}(\partial\Omega_u)\cap B_\rho(x)\subset \text{\rm Graph}(\varphi)\cap B_\rho(x),$$
and $\nabla\varphi(0)=0$ and 
$$|\nabla \varphi(x_0')-\nabla\varphi(y_0')|\le C\big(-\ln |x_0'-y_0'|\big)^{-\delta}\quad\text{for every}\quad x_0',y_0'\in \R^k\cap B_\rho(x).$$
\end{enumerate}
\end{theo}
\begin{proof} The claims (i) and (ii) follow by the log-epiperimetric inequality \eqref{e:sing:epi} and a standard general procedure, which we explain in the appendix (Proposition \ref{p:app:B}) for the sake of completeness. The claim (iii) follows directly from (ii), while (iv) is a consequence of (iii) and a Whitney extension theorem (see, for instance, \cite{fef}).
\end{proof}


\subsection{Acknowledgments} 
L.S. has been partially supported by the NSF grant DMS 1810645. B.V. has been supported by the European Research Council (ERC) under the European Union's Horizon 2020 research and innovation programme (grant agreement VAREG, No. 853404).

\section{Preliminaries}\label{s:prelimaries}

In this section we prove some preliminary results about the energy $W$ (Section \ref{sub:weiss}) and we recall some a basic facts about the decomposition in spherical harmonics (Section \ref{sub:fourier}) that we use in the construction of the competitors for the log-epiperimetric inequality.

\subsection{Properties of the Weiss' boundary adjusted energy}\label{sub:weiss}
In the direct proof of Theorem \ref{t:epi:sing} (see Section \ref{s:direct}) we will use only Lemma \ref{l:WW0}, Lemma \ref{l:slicing} and Corollary \ref{cor:orto}. Lemma \ref{l:slicing2} will be used in the second proof, given in Section \ref{s:constructive}.
\begin{lemma}[$W$ and $W_0$]\label{l:WW0}
Let $u\in H^1(B_1)$ and $Q\in \mathcal S$. Then
$$W_0(u-Q)=W(u)-W(Q).$$
\end{lemma}	
\begin{proof}
	We compute 
\begin{align}
W_0(u-Q)&=\int_{B_1}|\nabla (u-Q)|^2-2\int_{\partial B_1} (u-Q)^2\,d\HH^{d-1}\notag\\
&=W_0(u)+W_0(Q)-2\left(\int_{B_1}\nabla u\cdot\nabla Q\,dx-2\int_{\partial B_1} uQ\,d\HH^{d-1}\right)\notag\\
&=W_0(u)+W_0(Q)+2\int_{B_1}u\Delta Q\,dx\notag\\
&=W_0(u)+W_0(Q)+\int_{B_1}u\,dx=W(u)-W(Q),\label{e:energy_reduction}
\end{align}
where we integrated by parts and, in the last line, we used the identity 

$$W_0(Q)=\int_{B_1}|\nabla Q|^2\,dx-2\int_{\partial B_1}Q^2\,d\HH^{d-1}=-\int_{B_1}Q\Delta Q\,dx=-\frac12\int_{B_1}Q\,dx,$$
so that
$\ds W(Q)=W_0(Q)+\int_{B_1}Q\,dx=\frac12\int_{B_1}Q\,dx=-W_0(Q)\,.$
\end{proof}

The following simple lemma will be fundamental in both proofs so we collect it here.
\begin{lemma}[Slicing Lemma]\label{l:slicing} Let $u=u(r,\theta)\in H^1(B_1)$. Then, computing the energy $ W_0(r^2u)$ of the function written in polar coordinates as $(r,\theta)\mapsto r^2u(r,\theta)$, we obtain
	\begin{equation}\label{e:slicing1}
	 W_0(r^2u)=\int_0^1 r^{d+1} \int_{\de B_1} \left(|\nabla_\theta u|^2  - 2d\,u^2\right)\,d\theta\,dr+\int_0^1 r^{d+3}\int_{\de B_1} |\de_r u|^2\,d\theta\,.
	\end{equation}
	In particular, if we set 
	\begin{equation}\label{e:mathcalF}
	\mathcal F(\phi):=\int_{\de B_1} \left(|\nabla_\theta \phi|^2  - 2d\,\phi^2+\phi\right)\,d\HH^{d-1}\,,
	\end{equation}
	we have the equality
	\begin{equation}\label{e:slicing2}
	W(r^2u)=\int_0^1 \mathcal F\big(u(r,\cdot)\big) r^{d+1} \,dr+\int_0^1r^{d+3}\int_{\de B_1} |\de_r u|^2\,d\HH^{d-1}\,dr.
	\end{equation}
	Finally,  if $u(r,\theta)=r^{\eps} c(\theta)$, then
	\begin{equation}\label{e:slicing_homo}
	W_0\big(r^{2+\eps} c(\theta)\big)= \frac{1}{d+2\alpha-2}\int_{\partial B_1}\Big(|\nabla_\theta c|^2-2d c^2\Big)d\theta+\frac{(\alpha-2)^2}{d+2\alpha-2}\int_{\partial B_1}c^2\,d\theta,
	\end{equation}
	where $\alpha=2+\eps$ and $\eps\ge 0$.
\end{lemma} 

\begin{proof}
	Setting $\theta\in\partial B_1$, $d\theta=d\HH^{d-1}$, we calculate for a function $u=u(r,\theta)$
	\begin{align} 
	 W_0(r^2u)
	&=\int_0^1 \int_{\de B_1}\left(|2\,r u+r^2 \de_r u|^2+r^{2}|\nabla_\theta u|^2\right)\,d\theta\,r^{d-1}\,dr-2\,\int_{\de B_1} u^2 \,d\theta\notag\\
	&=\int_0^1 \int_{\de B_1} \left( 2 r^{2}u^2+r^{4} |\de_r u|^2+ 2\, r^{3} \de_r(u^2)+r^{2} |\nabla_\theta u|^2 \right)\,d\theta\,r^{d-1}\,dr-2\,\int_{\de B_1} u^2 \,d\theta\notag\\
	&=\int_0^1 \int_{\de B_1} \left( 4 r^{2}u^2+r^{4} |\de_r u|^2- 2(d+2)\, r^{2} u^2+r^{2} |\nabla_\theta u|^2 \right)\,d\theta\,r^{d-1}\,dr\notag\\
	&=\int_0^1 r^{d+1} \int_{\de B_1} \left(|\nabla_\theta u|^2  - 2d\,u^2\right)\,d\theta\,dr+\int_0^1 r^{d+3}\int_{\de B_1} |\de_r u|^2\,d\theta\,dr,\notag
	\end{align}
	which is precisely \eqref{e:slicing1}. The identity \eqref{e:slicing2} follows from \eqref{e:slicing1} by the formula
	$$\ds\int_{B_1}u\,dx=\int_0^1\int_{\partial B_r}u\,d\HH^{d-1}\,dr.$$
	Finally, equation \eqref{e:slicing_homo} directly follows from \eqref{e:slicing1} by integrating in $r$.
\end{proof}

\begin{cor}[Decomposition of the energy]\label{cor:orto}
	Suppose that $z_1,z_2\in H^1(B_1)$ are of the form $z_j(r,\theta)=r^2g_j(r)c_j(\theta)$, for $j=1,2$, where the traces $c_1,c_2\in H^1(\partial B_1)$ are orthogonal on the sphere in the following sense:
	\begin{equation}\label{e:orto}
	\int_{\partial B_1}c_1c_2\,d\theta=\int_{\partial B_1}\nabla_\theta c_1\cdot\nabla_\theta c_2\,d\theta=0.
	\end{equation}
	Then, we have 
	\begin{equation}\label{e:orto:energy}
	W_0(z_1+z_2)=W_0(z_1)+W_0(z_2)\qquad\text{and}\qquad W(z_1+z_2)=W(z_1)+W(z_2).
	\end{equation}
\end{cor}	
\begin{proof}It is sufficient to apply the formulas \eqref{e:slicing1} and \eqref{e:slicing2}, and then use \eqref{e:orto}.
\end{proof}
The next lemma is essentially the identity \eqref{e:slicing2} from the Slicing Lemma for competitors defined by reparametrization of the radial coordinate. We will use the following notation:
\begin{equation}\label{e:nabla_mathcalF}
\nabla\mathcal F(\phi)=-2\Delta_{\partial B_1}\phi-4d\phi+1\qquad\text{for}\qquad \phi\in H^2(\partial B_1),
\end{equation}
$$\psi\cdot\nabla\mathcal F(\phi)=\int_{\partial B_1}\psi\big(-2\Delta_{\partial B_1}\phi-4d\phi+1\big)\,d\HH^{d-1}\qquad\text{for}\qquad \phi\in H^2(\partial B_1) \quad\text{and}\quad \psi\in L^2(\partial B_1).$$
\begin{lemma}[The slicing lemma reparametrized]\label{l:slicing2}
Suppose that $\psi:[0,+\infty)\to H^2(\partial B_1)$ is a function such that 
$$\psi\in C^1\big((0,+\infty);L^2(\partial B_1)\big)\cap C^0\big([0,+\infty);H^1(\partial B_1)\big)\cap C^0\big((0,+\infty);H^2(\partial B_1)\big),$$
and let $T\in(0,+\infty]$. We define $\varphi:[0,+\infty)\to H^2(\partial B_1)$ as 
$$\varphi(t):=\psi(t)\quad\text{if}\quad t\in[0,T),\qquad \varphi(t):=\psi(T)\quad\text{if}\quad t\ge T.$$
and the function $u:B_1\to \R$ as  
$$u(r,\theta)=\varphi(-\kappa\ln r,\theta),$$
where $\kappa>0$ is fixed. Then, we have 

\begin{equation}\label{e:slicing666}
W\big(r^2u\big)=\frac1\kappa\int_0^\infty \mathcal F(\varphi(t)) \, e^{-\frac{t(d+2)}\kappa}\,dt+\kappa\int_{0}^\infty\|\varphi'(t)\|_{L^2(\partial B_1)}^2   \,e^{-\frac{t(d+2)}\kappa}\,dt, 
\end{equation}	
and also
\begin{equation}\label{e:slicing33}
W\big(r^2u\big)=\frac{\mathcal F(\psi(0))}{d+2}+\int_0^T \left(\frac{1}{d+2}\nabla\mathcal F(\psi(t)) \cdot \psi'(t)+ \kappa\|\psi'(t)\|^{2}_{L^2(\partial B_1)}  \right)\,e^{-\frac{t(d+2)}\kappa}   \,dt.
\end{equation}	
In particular, if $\varphi$ is constant in $t$, then 
\begin{equation}\label{e:slicing44}
W\big(r^2u\big)=\frac{\mathcal F(\varphi(0))}{d+2}.
\end{equation}	
\end{lemma}	
\begin{proof}Using the identity \eqref{e:slicing2} and the change of variables $t=-\kappa \ln r$, we compute
	\begin{align*}\label{e:Loj1}
	W(r^2u)&=\int_0^1 \mathcal F\big(u(r)\big) r^{d+1} \,dr+\int_0^1r^{d+3}\int_{\de B_1} |\de_r u|^2\,d\HH^{d-1}\,dr\\
&=\int_0^1 \mathcal F\big(\varphi(-\kappa \ln r)\big) r^{d+1} \,dr+\int_0^1\kappa^2 r^{d+1}\|\varphi'(-\kappa \ln r)\|_{L^2(\partial B_1)}^2\,dr\\
	&=\frac1\kappa\int_0^\infty \mathcal F(\varphi(t)) \, e^{-\frac{t(d+2)}\kappa}\,dt+\kappa\int_{0}^\infty\|\varphi'(t)\|_{L^2(\partial B_1)}^2   \,e^{-\frac{t(d+2)}\kappa}\,dt\notag 
	\end{align*}
	Now, by the definition of $\psi$, we get 
		\begin{align*}
	W(r^2u)&=\frac1\kappa\int_0^T \mathcal F(\psi(t)) \,e^{-\frac{t(d+2)}\kappa} dt+\frac1\kappa\int_T^{+\infty} \mathcal F(\psi(T)) \,e^{-\frac{t(d+2)}\kappa} dt+\kappa\int_{0}^T\|\psi'(t)\|_{L^2(\partial B_1)}^2  e^{-\frac{t(d+2)}\kappa}\, dt \\
	&=\frac1\kappa\int_0^T \mathcal F(\psi(t)) \,e^{-\frac{t(d+2)}\kappa} dt+\frac{\mathcal F(\psi(T))}{d+2}e^{-\frac{T(d+2)}\kappa} +\kappa\int_{0}^T\|\psi'(t)\|_{L^2(\partial B_1)}^2  e^{-\frac{t(d+2)}\kappa}\, dt.
	\end{align*}
	Now, an integration by parts gives
		\begin{align*}
\frac1\kappa\int_0^T \mathcal F(\psi(t)) \,e^{-\frac{t(d+2)}\kappa}\,dt=\frac{1}{d+2}\int_0^T \psi'(t)\cdot\nabla\mathcal F(\psi(t)) \,e^{-\frac{t(d+2)}\kappa}  dt+\frac{\mathcal F(\psi(0))}{d+2} -e^{-\frac{T(d+2)}\kappa} \frac{\mathcal F(\psi(T))}{d+2},
	\end{align*}
which concludes the proof.
\end{proof}	
\subsection{Spectrum of the spherical Laplacian}\label{sub:fourier}
Let $0<\lambda_1\le \lambda_2\le \dots\le\lambda_j\le \dots$ be the eigenvalues (counted with multiplicity) of the spherical Laplace-Beltrami operator and $\{\phi_j\}_{j\ge1}$ be the corresponding family of eigenfunctions, that is the solutions of 
\begin{equation}
\label{defn:eigenval-in-s}
-\Delta_{\mathbb S^{d-1}} \phi_j=\lambda_j\phi_j\quad\text{on}\quad \mathbb S^{d-1},\qquad \int_{\mathbb S^{d-1}}\phi_j^2(\theta)\,d\theta=1.
\end{equation}
Then, for any fixed $i,j\in\N$, we have 
$$\int_{\partial B_1}\phi_i\phi_j\,d\theta=\delta_{ij}\qquad\text{and}\qquad \int_{\partial B_1}\nabla_\theta \phi_i\cdot \nabla_\theta\phi_j\,d\theta=\lambda_i\delta_{ij}.$$
It is well known that the eigenfunctions of the Spherical Laplacian on $\mathbb S^{d-1}$ are in fact the traces of homogeneous harmonic polynomials in $\R^d$. In fact, for a given $\alpha\ge 0$, it is immediate to check that a function $\phi:\partial B_1\to\R$ is an eigenfunction corresponding to the eigenvalue 
	$$\lambda(\alpha):=\alpha(\alpha+d-2),$$ 
	if and only if, its $\alpha$-homogeneous extension $\varphi(r,\theta)=r^{\alpha}\phi(\theta)$ is harmonic in $B_1$. We will denote by $\alpha_j$ the homogeneity corresponding to the eigenvalue $\lambda_j$, that is, we have
	\begin{equation}\label{e:alpha_j}
	\lambda_j=\lambda(\alpha_j)=\alpha_j(\alpha_j+d-2).
	\end{equation}
	Notice that, since the  homogeneous harmonic functions in $\R^d$ are polynomials, we have that $\alpha_j\in\N$. Thus, we can easily identify the eigenvalues and the eigenfunctions of the spherical Laplacian corresponding to the first few elements of the spectrum. Precisely, we have that :
	
	\medskip
	
	$\bullet$ $\alpha_1=\lambda_1=0$ and the corresponding eigenfunction is the constant $\phi_1=|\partial B_1|^{-\sfrac12}= (d\omega_d)^{-\sfrac12}$. 
	
	\medskip
	
	$\bullet$ $\lambda_2=\dots=\lambda_{d+1}=d-1$, the corresponding homogeneity is $\alpha_2=\dots=\alpha_{d+1}=1$, while the
	eigenspace  coincides with the ($d$-dimensional) space of linear functions in $\R^d$. 
	
	\medskip
	
	$\bullet$ $\lambda_{d+2}=\dots=\lambda_{\sfrac{d(d+3)}2}=2d$, the corresponding homogeneity is $\alpha_{d+2}=\dots=\alpha_{\sfrac{d(d+3)}2}=2$; the corresponding eigenspace $E_{2d}$ 
	is generated by the (restrictions to $\mathbb{S}^{d-1}$ of the) $2$-homogeneous harmonic polynomials:
	$$E_{2d}=\{Q_A \colon \R^d \to \R \,:\, Q_A(x) = x \cdot Ax,\,\text{ $A$  symmetric with }{\rm{tr}} A = 0\}.$$
	In particular, if $Q\in \mathcal S$ is an admissible singular blow-up limit, then $Q$ is of the form
	\begin{equation*}
	Q(x)=\frac1{4d}|x|^2+Q_A(x)\,\quad   \mbox{for some harmonic polynomial} \quad Q_A\in E_{2d}. 
	\end{equation*}
	$\bullet$ Finally, if $j> \frac{d(d+3)}2$ (that is $\lambda_j>2d$), then  the corresponding homogeneity is at least $3$ and so 
	$$\lambda_j\ge 3(3+d-2)=3(d+1).$$ 


\subsection{Decomposition of the trace and energy of the $2$-homogeneous extension}\label{sub:decomposition}
Let $c$ be as in Theorem \ref{t:epi:sing}. Throughout the rest of the paper we will use the same decomposition of the trace with the same notation. We will denote by $Q\in\mathcal S$ the projection of $c$ on the set of critical points $\mathcal S$ given in \eqref{e:critical_points}. Precisely, $Q$ realizes the minimum
$$\dist_{2}(c,\mathcal S):=\min\big\{\|c-Q\|_{L^2(\partial B_1)}\ :\ Q\in\mathcal S \big\}.$$
We now decompose the function $c-Q\in H^1(\partial B_1)$ in Fourier series as  
$$\ds c(\theta)-Q(\theta)=\sum_{j=1}^{\infty}c_j\phi_j(\theta)$$
where $c_j$ are the Fourier coefficients
\begin{equation}\label{e:coef_fou}
\ds c_j:=\int_{\partial B_1}\big(c(\theta)-Q(\theta)\big)\,\phi_j(\theta)\,d\theta.
\end{equation}
Finally, we will write $c:\partial B_1\to\R$ as 
\begin{equation}\label{e:decomposition_of_c}
c=Q+\eta_++\eta_0+\eta_-,
\end{equation}
where the functions $\eta_+$, $\eta_0$ and $\eta_-$ (defined on $\partial B_1$) are given by 
$$\ds \eta_-:=\sum_{j\,:\, \alpha_j< 2}c_j\phi_j\ ,\qquad \eta_0:=\sum_{j\,:\, \alpha_j=2}c_j\phi_j\qquad\text{and}\qquad \eta_+:=\sum_{j\,:\, \alpha_j>2}c_j\phi_j\ .$$

\subsection{Energy of the $2$-homogeneous extension $z$}\label{sub:decomposition}
In terms of the decomposition \eqref{e:decomposition_of_c}, $2$-homogeneous extension of $c(\theta)$ can be written as:
\begin{equation}\label{e:decomposition_z}
z(r,\theta)=r^2c(\theta)=Q(r\theta)+r^2\eta_-(\theta)+r^2\eta_0(\theta)+r^2\eta_+(\theta),
\end{equation}
where we recall that $Q\in\mathcal S$ is $2$-homogeneous: 
$$Q(r\theta)=r^2Q(\theta)\quad\text{for every}\quad r>0\quad\text{and}\quad \theta\in\partial B_1.$$
Notice that the functions $\eta_-$, $\eta_0$ and $\eta_+$ are orthogonal on the sphere in the sense that 
$$\int_{\partial B_1}\eta_i(\theta)\eta_j(\theta)\,d\theta=\int_{\partial B_1}\nabla_\theta\eta_i(\theta)\cdot \nabla_\theta\eta_j(\theta)\,d\theta=0\qquad\text{whenever}\qquad i\neq j\in\{+,-,0\}.$$
Thus, by Lemma \ref{l:WW0} and Corollary \ref{cor:orto}, we can compute the term in the right-hand side of the log-epiperimetric inequality in terms of $\eta_+$, $\eta_0$ and $\eta_-$. 
\begin{equation}\label{e:decomposition_z}
W(z)-W(Q)=W_0(z-Q)=W_0\big(r^2\eta_-(\theta)\big)+W_0\big(r^2\eta_0(\theta)\big)+W_0\big(r^2\eta_+(\theta)\big).
\end{equation}

\section{The best direction and a \L ojasiewicz-type inequality}\label{s:best}
Let $c=Q+\eta_++\eta_0+\eta_-$ be as in Section \ref{sub:decomposition}. Let $M$ be the maximum of the negative part of $\eta_-+\eta_0+Q$, that is,  

\begin{equation}\label{e:def:M}
M:=\max_{\theta\in \partial B_1}\big\{-\eta_-(\theta)-\eta_0(\theta)-Q(\theta)\big\}.
\end{equation}
We define the functions $h_2:\partial B_1\to\R$ and $h_\alpha:\partial B_1\to\R$ as
\begin{equation}\label{e:h2halpha}
h_2:=Q+\eta_-+\eta_0+8dM\left(\frac1{4d}-Q\right)\qquad \text{and}\qquad h_\alpha:=\eta_+-8dM\left(\frac1{4d}-Q\right)\,.
\end{equation}
The role of the correction term $\ds 8dM\left(\frac1{4d}-Q\right)$ will be explained in Section \ref{sub:direct}. In the lemma below we gather the key estimates, which we will use in both proofs of Theorem \ref{t:epi:sing} - the one based on the direct construction of the competitor (Section \ref{s:direct}) and the one based on the definition of a flow and a \L ojasiewicz-type inequality (Section \ref{s:constructive}).

\begin{lemma}[Key Estimate]\label{p:key} Let $c$, $Q$, $\eta_-$, $\eta_0$, $\eta_+$, $M$, $h_2$ and $h_\alpha$ be as above :
	$$c(\theta)=Q(\theta)+\eta_-(\theta)+\eta_0(\theta)+\eta_+(\theta)=h_2(\theta)+h_\alpha(\theta),\quad \theta\in \partial B_1.$$ 
	Then, there is a dimensional constant $\delta>0$ such that the following holds. If 
	$$\|c-Q\|_{L^2(\partial B_1)}\le \delta,$$ 
then we have:
	\begin{itemize}
	\item[(i)] $h_2(\theta)\geq 0\ $ for every $\ \theta\in\partial B_1$.
	\medskip
	\item[(ii)] there is a dimensional constant $C_d>0$ such that 
	\begin{equation}\label{e:estimate_M}
	M^{d+1}\le C_d\|\eta_+\|_{L^2(\partial B_1)}^2\ ;
	\end{equation}
	\item[(iii)] for every $t\in\R$, we have the  following identities : 
		\begin{equation}\label{e:new_loja}
		h_\alpha\cdot \nabla \mathcal F(h_2+th_\alpha) =2t\int_{\partial B_1}\Big(|\nabla_\theta \eta_+|^2-2d\,\eta_+^2\Big)\,d\theta\,,
		\end{equation}	
		
			\begin{equation}\label{e:new_loja_F}
	 \mathcal F(h_2+th_\alpha) =\mathcal F(Q)+\int_{\partial B_1}\Big(|\nabla_\theta \eta_-|^2-2d\,\eta_-^2\Big)\,d\theta+t^2\!\!\int_{\partial B_1}\Big(|\nabla_\theta \eta_+|^2-2d\,\eta_+^2\Big)\,d\theta\,,
		\end{equation}	
		where we recall that $\mathcal F$ and $\nabla \mathcal F$ are given by \eqref{e:mathcalF} and \eqref{e:nabla_mathcalF}.
	\end{itemize}
\end{lemma}


\begin{proof} We start by proving (i).\\
Notice that there is a dimensional constant $C_d$ such that 
$$\|\phi_j\|_{L^\infty(\partial B_1)}+\|\nabla_\theta \phi_j\|_{L^\infty(\partial B_1)}\le C_d\quad\text{for every}\quad j\in\N\quad\text{such that}\quad \alpha_j\le 2d. $$
Now, since by definition 
$$\eta_-(\theta)+\eta_0(\theta)=\sum_{j\,:\,\alpha_j\le 2d} c_j \phi_j(\theta),$$
we can find another dimensional constant $C>0$ such that 
\begin{align*}
\|\eta_-+\eta_0\|_{L^\infty}&\le \sum_{j\,:\,\alpha_j\le 2d} |c_j|\|\phi_j\|_{L^\infty(\partial B_1)}\\
&\le C\Big(\sum_{j\,:\,\alpha_j\le 2d} c_j^2\ \Big)^{\sfrac12}= C\|\eta_-+\eta_0\|_{L^2(\partial B_1)}\le C\,\|c-Q\|_{L^2(\partial B_1)}\,,
\end{align*}
where we recall that 
$$\|c-Q\|_{L^2(\partial B_1)}=\|\eta_-+\eta_0+\eta_+\|_{L^2(\partial B_1)}.$$


\noindent We now choose $\delta>0$ such that $\ds 4C\delta\le \frac{1}{4d}$. We next show that $h_2\ge 0$ on each of the sets 
$$\{\theta\in \partial B_1\ :\ Q(\theta)\ge 2C\delta\}\qquad \text{and}\qquad \{\theta\in \partial B_1\ :\ Q(\theta)\le 2C\delta\}.$$
Indeed, we have the following two cases.
\begin{itemize}
	\item Consider the set $\{Q\ge 2C\delta\}\subset\partial B_1$. We first notice that 
	$$\eta_-+\eta_0+\frac12Q\ge 0$$ 
	on this set. Indeed, for any $\theta\in \{Q\ge 2C\delta\}$, we have 
	\begin{align*}
	\eta_-(\theta)+\eta_0(\theta)+\frac12Q(\theta)&\ge -\|\eta_-+\eta_0\|_{L^\infty(\partial B_1)}+\frac12Q(\theta)\ge -C\,\|c-Q\|_{L^2(\partial B_1)}+C \delta \ge 0.
	\end{align*}
	We now decompose $h_2$ as follows :
	$$h_2=\left(\eta_-+\eta_0+\frac12Q\right)+2M+Q\left(\frac12-8dM\right),$$
	where the first and the second terms are nonnegative. In order to prove that also the third one is nonnegative, we notice that since $Q\ge 0$, we have 
	\begin{align*}
	M\le \|\eta_-+\eta_0\|_{L^\infty(\partial B_1)}\le C \delta\,,
	\end{align*}
	so, by the choice of $\delta$, 
	$$\frac12-8dM\ge \frac12-8dC\delta=2d\left(\frac1{4d}-4C\delta\right)\ge 0.$$
This concludes the first part of the proof of (i).\medskip 
	
	\item Consider the set $\{Q\le 2C\delta\}\subset\partial B_1$. We have that
	\begin{align*}
	h_2=\eta_-+\eta_0+Q+8dM\left(\frac1{4d}-Q\right)&\ge  \eta_-+\eta_0+Q+8dM\left(\frac1{4d}-2C\delta\right)\\
	&\ge   \eta_-+\eta_0+Q+M\ge 0,
	\end{align*}
	where the last inequality is due to the choice \eqref{e:def:M} of $M$. This concludes the proof of (i).
\end{itemize}
\medskip

Next we prove (ii). We set for simplicity 
$$P:=\eta_-+\eta_0+Q$$
and we notice that $P$ is a polynomial of degree two (a linear combination of eigenfunctions corresponding to eigenvalues $\le 2d$). 
We claim that there is a dimensional constant $L$ such that $\|\nabla P\|_{L^\infty(\partial B_1)}\le L$. Indeed, reasoning as in the proof of (i), we have that there is a dimensional constant $C_d$ such that
 \begin{align*}
 \|\nabla_\theta \eta_-+\nabla_\theta \eta_0\|_{L^\infty(\partial B_1)}&\le \sum_{j\,:\,\alpha_j\le 2d} |c_j|\|\nabla_\theta \phi_j\|_{L^\infty(\partial B_1)}\le C_d\Big(\sum_{j\,:\,\alpha_j\le 2d} c_j^2\ \Big)^{\sfrac12}\\
 &= C_d\|\eta_-+\eta_0\|_{L^2(\partial B_1)}\le C_d\,\|c-Q\|_{L^2(\partial B_1)}\le C_d\,\delta\,.
 \end{align*} 
On the other hand, by the definition of $\mathcal S$ (see \eqref{e:critical_points}), $Q$ is of the form 
$$Q(x)=x\cdot Ax\quad\text{for a nonnegative symmetric matrix}\quad A\quad \text{with}\quad {\rm{tr}} A = \frac14.$$
Thus, for the gradient of $Q$, we have  
$$\|\nabla Q\|_{L^\infty}=\sup_{x\neq 0}\frac{2|Ax|}{|x|}\le 2\,\text{\rm tr}\, A=\frac12.$$
In particular, 
$$\|\nabla_\theta P\|_{L^\infty(\partial B_1)}\le C_d\delta+\frac12.$$
Now, since $Q$ is non-negative and $\eta_-+\eta_0$ is small, 
we have that the negative part $\inf\{P,0\}$ is also small. Precisely, 
$$M=\|\inf\{P,0\}\|_{L^\infty(\partial B_1)}\le  \|\eta_-+\eta_0\|_{L^\infty(\partial B_1)}\le C\delta.$$
Since the function $ \inf\{P,0\}:\partial B_1\to\R$ is $L$-Lipschitz and is small (in $L^\infty(\partial B_1)$), we can deduce that there is a dimensional constant $C_d>0$ such that
$$\|\inf\{P,0\}\|_{L^2(\partial B_1)}^2\ge C_d\, \frac{M^{d-1}}{L^{d-1}} M^2=C_d\, L^{-d+1}\|\inf\{P,0\}\|_{L^\infty(\partial B_1)}^{d+1}=C_d\, L^{-d+1}M^{d+1},$$
where the first inequality follows from Lemma \ref{l:revised_version}. Indeed, let $F:\R^{d-1}\to\R$ be the function $\inf\{P,0\}$ written in local coordiantes around the point, where its maximum is achieved on the sphere. In this local chart, we can use Lemma \ref{l:revised_version} to estimate the integral of $F^2$ over $B_R$, where $R:=\frac1L\|\inf\{P,0\}\|_{L^\infty}$.
On the other hand, the trace $c:\partial B_1\to\R$ is non-negative : 
$$c(\theta):=P(\theta)+\eta_+(\theta)\ge 0\qquad\text{for every}\quad  \theta\in\partial B_1.$$
Thus, necessarily,  $$-\inf\{P,0\}\le |\eta_+|\quad\text{on}\quad \partial B_1,$$ and so, 
$$C_d\, L^{-d+1}M^{d+1}\le \|\inf\{P,0\}\|_{L^2(\partial B_1)}^2\le \|\eta_+\|_{L^2(\partial B_1)}^2,$$
which (since $L$ is a dimensional constant) gives \eqref{e:estimate_M}.
\bigskip

Finally we prove (iii). 
We start with computing $\nabla \mathcal F(h_2)$.
First, we notice that, since $Q\in \mathcal S$, it is a solution to the PDE
\begin{equation}\label{e:equation_for_Q}
\nabla \mathcal F(Q)=-2\Delta_{\partial B_1}Q-4dQ+1=0.
\end{equation}
As a consequence of \eqref{e:equation_for_Q} and the definition of $\eta_0$, we have that both  $\eta_0$ and $\ds\Big(\frac1{4d}-Q\Big)$ are eigenfunctions of the spherical Laplacian, corresponding to the eigenvalue $2d$, that is, 
\begin{equation}\label{e:gru}
-\Delta_{\partial B_1}\eta_0-2d\eta_0=-\Delta_{\partial B_1}\Big(\frac1{4d}-Q\Big)-2d\Big(\frac1{4d}-Q\Big)=0.
\end{equation}
Thus,  we have
$$\nabla \mathcal F(h_2)= -2\Delta_{\partial B_1} h_2-4d h_2+1=-2\Delta_{\partial B_1} \eta_- -4d \eta_-\,.$$
Analogously, we compute $\nabla\mathcal F(h_\alpha)$. Indeed, we have
\begin{align*}
\nabla\mathcal F(h_\alpha)&=-2\Delta h_\alpha-4d \,h_\alpha\\
&=-2\Delta \left(\eta_+-8dM\Big(\frac1{4d}-Q\Big) \right) -4d \left(\eta_+-8dM\Big(\frac1{4d}-Q\Big) \right)\\
&=-2\Delta \eta_+ -4d \eta_+ 
\end{align*}
where in the last equality we used again  \eqref{e:gru}. 

In order to compute $h_\alpha\cdot\nabla \mathcal F(h_2+t\, h_\alpha)$, we first write $h_\alpha$ in the form 
$$h_\alpha=\eta_++\widetilde \eta_0\qquad\text{where}\qquad \widetilde \eta_0:=-8dM\Big(\frac1{4d}-Q\Big),$$
and we notice that by \eqref{e:gru}, $\widetilde \eta_0$ is a $(2d)$-eigenfunction of the spherical Laplacian. 
Using the definition of $\nabla\mathcal F$ \eqref{e:nabla_mathcalF} and the fact that $\nabla \mathcal F(Q)=0$ \eqref{e:equation_for_Q}, we compute 
\begin{align*}
h_\alpha\cdot\nabla \mathcal F(h_2+t\, h_\alpha)&=
\big(\eta_++\widetilde\eta_0\big)\cdot\nabla \mathcal F\big(Q+\eta_-+\eta_0-\widetilde\eta_0+t\, \eta_++t\,\widetilde\eta_0\big)\\
&=
\big(\eta_++\widetilde\eta_0\big)\cdot\nabla \mathcal F\big(\eta_-+\eta_0-\widetilde\eta_0+t\, \eta_++t\,\widetilde\eta_0\big)\\
&=
\big(\eta_++\widetilde\eta_0\big)\cdot\Big[\big(-2\Delta-4d\big)(\eta_-+\eta_0-\widetilde\eta_0+t\, h_\alpha+t\,\widetilde\eta_0\big)+1\Big].
\end{align*}
Now, since $\int_{\partial B_1}(\eta_++\widetilde \eta_0)=0$ and since both $\eta_0$ and $\widetilde \eta_0$ are $(2d)$-eigenfunctions, we get 
\begin{align*}
h_\alpha\cdot\nabla \mathcal F(h_2+t\, h_\alpha)
&=
\big(\eta_++\widetilde\eta_0\big)\cdot\big(-2\Delta-4d\big)(\eta_-+t\, \eta_+\big).
\end{align*}
Next, notice that by the definition of $\eta_-,\ \eta_+$ and $\widetilde \eta_0$, they are orthogonal in $L^2(\partial B_1)$ and $H^1(\partial B_1)$: 
$$\int_{\partial B_1}\eta_+(\theta)\widetilde \eta_0(\theta)\,d\theta=\int_{\partial B_1}\eta_-(\theta)\widetilde \eta_0(\theta)\,d\theta=\int_{\partial B_1}\eta_+(\theta) \eta_-(\theta)\,d\theta=0,$$
$$\int_{\partial B_1}\nabla_\theta\eta_+\cdot \nabla_\theta\widetilde \eta_0\,d\theta=\int_{\partial B_1}\nabla_\theta\eta_-\cdot \nabla_\theta\widetilde \eta_0\,d\theta=\int_{\partial B_1}\nabla_\theta\eta_+\cdot\nabla_\theta\eta_-\,d\theta=0.$$
Thus, integrating by parts on $\partial B_1$, we obtain 
\begin{align*}
h_\alpha\cdot\nabla \mathcal F(h_2+t\, h_\alpha)
	&=2t\int_{\partial B_1}\Big(|\nabla_\theta \eta_+|^2-2d\eta_+^2\Big)\,d\theta,
\end{align*}
as required. It remains to prove \eqref{e:new_loja_F}. Using \eqref{e:equation_for_Q}, we compute
\begin{align*}
\mathcal F(h_2+t\, h_\alpha)-\mathcal F(Q)&=\mathcal F\big(Q+\eta_-+\eta_0+t\eta_++(t-1)\widetilde \eta_0\big)-\mathcal F(Q)\\
&=\int_{\partial B_1}\big|\nabla_\theta \big(\eta_-+\eta_0+t\eta_++(t-1)\widetilde \eta_0\big)\big|^2\,d\theta\\
&\qquad\qquad -2d\int_{\partial B_1}\big(\eta_-+\eta_0+t\eta_++(t-1)\widetilde \eta_0\big)^2\,d\theta\\
&=\int_{\partial B_1}\Big(|\nabla_\theta \eta_-|^2-2d\eta_-^2\Big)\,d\theta+t^2\int_{\partial B_1}\Big(|\nabla_\theta \eta_+|^2-2d\eta_+^2\Big)\,d\theta,
\end{align*}
where the last equality follows from the orthogonality (in $L^2(\partial B_1)$ and $H^1(\partial B_1)$) of $\eta_+$, $\eta_-$ and $\eta_0+(t-1)\widetilde \eta_0$, and from the fact that $\eta_0+(t-1)\widetilde \eta_0$ is an eigenfunction of the spherical Laplacian corresponding precisely to the eigenvalue $2d$.
\end{proof}
%

\begin{lemma}\label{l:revised_version}
Suppose that $n\ge1$ and that $F:\R^n\to\R$	 is a function which is nonnegative a
and $L$-Lipschitz continuous for some constant $L>0$. Let $x_0\in\R^n$ and let $M:=F(x_0)>0$. 
Then,
\begin{equation}\label{e:estimate_with_constants}
\int_{B_R(x_0)}F^2(x)\,dx\ge \frac{2\omega_n}{(n+1)(n+2)}\frac{M^{n+2}}{L^n},
\end{equation}
 where $\omega_n$ is the volume of the unit ball in $\R^n$ and $R=\sfrac{M}{L}$.
\end{lemma}	
\begin{proof}
First notice that the $L$-Lipschit continuity of $F$ implies that 
$$F(x)\ge M-L|x-x_0|\ge 0\quad\text{for every}\quad x\in B_R(x_0).$$
Thus, integrating over $B_R(x_0)$, we get that 
\begin{align*}
\int_{B_R(x_0)}F^2(x)\,dx&\ge \int_{B_R(x_0)}\big(M-L|x-x_0|\big)^2\,dx.
\end{align*}	
Integrating the right-hand side in polar coordiantes and using the definition of $R$, we obtain
\begin{align*}
\int_{B_R(x_0)}F^2(x)\,dx&\ge n\omega_n\int_0^R\Big(M^2-2LMr+L^2r^2\Big)^2\,dr\\
&= n\omega_n\left(\frac1nR^nM^2-\frac{2}{n+1}LMR^{n+1}+\frac{1}{n+2}L^2R^{n+2}\right),
\end{align*}	
which is precisely \eqref{e:estimate_with_constants}.
\end{proof}	
\section{A direct proof of Theorem \ref{t:epi:sing}}\label{s:direct}

In this section we prove Theorem \ref{t:epi:sing} by giving the competitor explicitly, as in Subsection \ref{sub:direct}. 

\begin{proof}[Proof of Theorem \ref{t:epi:sing}]
We decompose the trace $c:\partial B_1\to\R$ as 
$$c=Q+\eta_++\eta_0+\eta_-\ ,$$
as in Subsection \ref{sub:decomposition} and we recall that the $2$-homogeneous extension $z$ is given by \eqref{e:decomposition_z}.
\medskip

\noindent{\it Definition of the competitor.}
We define the competitor $h:B_1\to\R$ as
\begin{equation}\label{e:competitor}
h(r,\theta):=r^2h_2(\theta)+r^\alpha h_\alpha(\theta),
\end{equation}
where $\alpha:=(2+\eps)> 2$, the functions $h_2$ and $h_\alpha$ are given by \eqref{e:h2halpha} 
as in Section \ref{s:best}.
\medskip

\noindent{\it Positivity of the competitor.} We first notice that the competitor $h$ defined in \eqref{e:competitor} is non-negative. Indeed, we can write the competitor $h$ as 
\begin{align*}
h(r,\theta)&=r^2h_2(\theta)+r^\alpha h_\alpha(\theta)=(r^2-r^\alpha)h_2(\theta)+r^\alpha c(\theta).
\end{align*}
Now, the first term $(r^2-r^\alpha)h_2(\theta)$ is non-negative by Lemma \ref{p:key} and the fact that $r\le 1$; the second term $r^\alpha c(\theta)$ is non-negative since the trace $c$ is non-negative by hypothesis. 

\medskip

\noindent{\it Decomposition of the energy.}
We first decompose the energy of $z$. We recall \eqref{e:decomposition_z} and we set
$$z_-(r,\theta):=r^2\eta_-(\theta),\qquad z_0(r,\theta):=r^2\eta_0(\theta)\quad\text{and}\quad z_+(r,\theta):=r^2\eta_+(\theta).$$
Since $\eta_-$, $\eta_0$ and $\eta_+$ are orthogonal, we have
$$W(z)-W(Q)=W_0(z-Q)=W_0(z_-)+W_0(z_0)+W_0(z_+).$$ 
We now estimate $W_0(z_-)$, $W_0(z_0)$ and $W_0(z_+)$. By \eqref{e:orto:energy} and \eqref{e:slicing_homo}, we have
\begin{align*}
W_0(z_0)=0\qquad\text{and}\qquad W_0(z_-)&=\sum_{j\,:\,\alpha_j<2}c_j^2W_0\big(r^2\phi_j(\theta)\big)=\frac1{d+2}\sum_{j\,:\,\alpha_j<2}c_j^2(\lambda_j-2d)\le 0.
\end{align*}
On the other hand, for the higher modes, we have 
\begin{align}
W_0(z_+)&=\sum_{j\,:\,\alpha_j>2}c_j^2W_0\big(r^2\phi_j(\theta)\big)=\frac1{d+2}\sum_{j\,:\,\alpha_j>2} c_j^2 (\lambda_j-2d)\notag\\
&\ge\frac1{3(d+2)}\sum_{j\,:\,\alpha_j>2} c_j^2 (\lambda_j+1)=\frac{1}{3(d+2)}\|\eta_+\|_{H^1(\partial B_1)}^2\ge 0\label{e:z+energy},
\end{align}
where we used that, if $\alpha_j>2$, then $\alpha_j\ge 3$ and $\lambda_j\ge 3(d+1)$. 

We now study $W(h)$, where $h$ is the competitor from \eqref{e:competitor}. On the other hand, setting 
$$\ds Q_0:=8d\Big(\frac1{4d}-Q\Big),$$ 
we get 
\begin{align*}
W(h)-W(Q)=W_0(h-Q)&=W_0(z_-)+W_0\big(z_0+(r^2-r^\alpha)MQ_0\big)+W_0\big(r^\alpha \eta_+(\theta)\big)\\
&=W_0(z_-)+W_0(z_0)+M^2W_0\big((r^2-r^\alpha)Q_0\big)+W_0\big(r^\alpha \eta_+(\theta)\big),
\end{align*}
where in the last equality we used that $z_0$ is harmonic in $B_1$ and $(r^2-r^\alpha)Q_0(\theta)$ vanishes on $\partial B_1$. Using the fact that $Q_0$ is a $2d$-eigenfunction and $\|Q_0\|_{L^2(\partial B_1)}\le C_d$, we calculate
\begin{align*}
W_0\big((r^2-r^\alpha)Q_0\big)&=\int_0^1 r^{d-1}\,dr\int_{\partial B_1}\Big((2r-\alpha r^{\alpha-1})^2Q_0^2+(r- r^{\alpha-1})^2|\nabla_\theta Q_0|^2\Big)\,d\theta\\
&=\|Q_0\|_{L^2(\partial B_1)}^2\int_0^1 r^{d+1}\Big(\big(2-\alpha r^{\alpha-2}\big)^2+2d\big(1- r^{\alpha-2}\big)^2\Big)\,dr\\
&=\|Q_0\|_{L^2(\partial B_1)}^2\frac{(\alpha-2)^2}{d+2\alpha-2}\le C_d(\alpha-2)^2 =C_d\, \eps^2.
\end{align*}
Putting together this estimate, \eqref{e:slicing_homo} and \eqref{e:z+energy}, we get 
\begin{align}
W_0(h-Q)-W_0(z-Q)&=W_0\big(r^\alpha \eta_+(\theta)\big)-W_0\big(r^2\eta_+(\theta)\big)+C_dM^2\eps^2\notag\\
&\le -\frac{\eps}{d+2}W_0(z_+)+\eps^2\|\eta_+\|_{L^2(\partial B_1)}^2+C_dM^2\eps^2
.\label{e:epi:main_est}
\end{align}

\noindent{\it Conclusion of the proof.}
We are finally in position to prove \eqref{e:sing:epi}. We first notice that by \eqref{e:epi:main_est} and  Lemma \ref{p:key} (ii), we have 
$$W_0(h-Q)-W_0(z-Q)\le -\frac{\eps}{d+2}W_0(z_+)+\eps^2\|\eta_+\|_{L^2(\partial B_1)}^2+C_d\|\eta_+\|_{L^2(\partial B_1)}^{\frac{2}{d+1}}\eps^2.$$
Recall that $\|\eta_+\|_{L^2(\partial B_1)}\le \delta$. Choosing $\delta\le 1$, we have 
\begin{align*}
W_0(h-Q)-W_0(z-Q)&\le -\frac{\eps}{d+2}W_0(z_+)+2C_d\|\eta_+\|_{L^2(\partial B_1)}^{\frac{2}{d+1}}\eps^2\\
&\le -\frac{\eps}{d+2}W_0(z_+)+C_dW_0(z_+)^{\frac{2}{d+1}}\eps^2,
\end{align*}
where in the last inequality we used \eqref{e:z+energy}. Finally, setting $\eps=C_dW_0(z_+)^{\frac{d-1}{d+1}}$, for some dimensional constant $C_d$, and using the inequality $W_0(z-Q)\le W_+(z_+)$, we get 
$$W_0(h-Q)-W_0(z-Q)\le -C_d \big(W_0(z-Q)\big)^{\frac{2d}{d+1}},$$
which is precisely \eqref{e:sing:epi} (see Lemma \ref{l:WW0}).
\end{proof}

\section{Constructive approach. A second proof of Theorem \ref{t:epi:sing}}\label{s:constructive}
Let $c$ be given by Theorem \ref{t:epi:sing}. We will use the general construction from Theorem \ref{p:epik}. In our case the homogeneity $\alpha$ is $2$, the functional $\mathcal G$ is the Weiss' boundary adjusted energy $W$, $\mathcal F$ is given by \eqref{e:mathcalF}, and the set $\mathcal S$ is \eqref{e:critical_points}. Thus, it only remains to define a flow $\psi$ that satisfies the energy dissipation inequality \eqref{e:energy_dissipation_condition} and the \L ojasiewicz inequality \eqref{e:lojaK}.

\subsection{Choice of the flow} We write the trace $c:\partial B_1\to\R$ as
$$c(\theta)=h_2(\theta)+h_\alpha(\theta),$$
where $h_2$ and $h_\alpha$ are given by \eqref{e:h2halpha}. We define the flow $\psi$ as
$$\psi(t)=h_2+e^{-t}h_\alpha\qquad\text{for every}\qquad t\ge 0.$$
Below we verify that $\psi$ satisfies the hypotheses of Proposition \ref{p:epik}. 

\subsection{\L ojasiewicz inequality} We first prove that the \L ojasiewicz inequality \eqref{e:lojaK} holds along the flow. Indeed, by \eqref{e:new_loja}, we have 
\begin{align*}
-\psi'(t)\cdot\nabla \mathcal F(\psi(t))&=e^{-t}h_\alpha\cdot\nabla \mathcal F(h_2+e^{-t}h_\alpha)=2e^{-2t}\int_{\partial B_1}\Big(|\nabla_\theta \eta_+|^2-2d\,\eta_+^2\Big)\,d\theta\,.
\end{align*}
On the other hand, \eqref{e:new_loja_F} implies that 
\begin{equation*}
\mathcal F(h_2+e^{-t}h_\alpha) -\mathcal F(\mathcal S)=\int_{\partial B_1}\Big(|\nabla_\theta \eta_-|^2-2d\,\eta_-^2\Big)\,d\theta+e^{-2t}\int_{\partial B_1}\Big(|\nabla_\theta \eta_+|^2-2d\,\eta_+^2\Big)\,d\theta\,.
\end{equation*}	
Finally, since 
$$\int_{\partial B_1}\Big(|\nabla_\theta \eta_-|^2-2d\,\eta_-^2\Big)\,d\theta\le 0,$$
we obtain the \L ojasiewicz inequality \eqref{e:lojaK} with constant $C_{\text{\sc ls}}=1$ and exponent $\beta=0$.

\subsection{Energy dissipation} In order to prove that \eqref{e:energy_dissipation_condition} holds, we compute 
\begin{align*}
\|\psi'(t)\|_{L^2(\partial B_1)}^2&=e^{-2t}\|h_\alpha\|_{L^2(\partial B_1)}^2=e^{-2t}\left(\|\eta_+\|_{L^2(\partial B_1)}^2+(8dM)^2\left\|\left(\frac1{4d}-Q\right)\right\|_{L^2(\partial B_1)}^2\right),
\end{align*}
where we used the fact that $\frac1{4d}-Q$ is an eigenfunction of the Spherical Laplacian corresponding to the eigenvalue $2d$ and so it is orthogonal to $\eta_+$ in $L^2(\partial B_2)$. In conclusion, since $Q\in\mathcal S$ and all the functions in $\mathcal S$ are bounded, we get that there is a dimensional constant $C_d$ such that  
\begin{align*}
\|\psi'(t)\|_{L^2(\partial B_1)}^2&\le e^{-2t}\left(\|\eta_+\|_{L^2(\partial B_1)}^2+C_d M^2\right)\\
&\le e^{-2t}\left(\|\eta_+\|_{L^2(\partial B_1)}^2+C_d \|\eta_+\|_{L^2(\partial B_1)}^{\frac{4}{d+1}}\right),
\end{align*}
where the second inequality follows from \eqref{e:estimate_M}. Now, since the Fourier decomposition of $\eta_+$ contains only eigenfunctions corresponding to eigenvalues $\lambda_j\ge 3(d+1)>2d$, we get  
$$\|\eta_+\|_{L^2(\partial B_1)}^2\le \int_{\partial B_1}\Big(|\nabla_\theta \eta_+|^2-2d\,\eta_+^2\Big)\,d\theta.$$
Thus, we consider the following two cases :
\begin{itemize}
\item If $\|\eta_+\|_{L^2(\partial B_1)}\ge 1$, then  	
\begin{align}
\|\psi'(t)\|_{L^2(\partial B_1)}^2&\le e^{-2t}\left(\|\eta_+\|_{L^2(\partial B_1)}^2+C_d \|\eta_+\|_{L^2(\partial B_1)}^{\frac{4}{d+1}}\right)\notag\\
&\le e^{-2t}(1+C_d)\|\eta_+\|_{L^2(\partial B_1)}^2\notag\\
&\le  e^{-2t}(1+C_d)\int_{\partial B_1}\Big(|\nabla_\theta \eta_+|^2-2d\,\eta_+^2\Big)\,d\theta\notag\\
&\le \frac{1+C_d}2\big(-\psi'(t)\cdot\nabla \mathcal F(\psi(t))\big).\label{e:constructive_conclusion_1}
\end{align}
\item Conversely, if $\|\eta_+\|_{L^2(\partial B_1)}\le 1$, then 
\begin{align}
\|\psi'(t)\|_{L^2(\partial B_1)}^2&\le e^{-2t}\left(\|\eta_+\|_{L^2(\partial B_1)}^2+C_d \|\eta_+\|_{L^2(\partial B_1)}^{\frac{4}{d+1}}\right)\notag\\
&\le e^{-2t}(1+C_d)\left(\|\eta_+\|_{L^2(\partial B_1)}^2\right)^{\frac{2}{d+1}}\notag\\
&\le (1+C_d)\left(e^{-2t}\int_{\partial B_1}\Big(|\nabla_\theta \eta_+|^2-2d\,\eta_+^2\Big)\,d\theta\right)^{\frac{2}{d+1}}\notag\\
&\le \frac{1+C_d}{2^{{\frac{2}{d+1}}}}\big(-\psi'(t)\cdot\nabla \mathcal F(\psi(t))\big)^{\frac{2}{d+1}}.\label{e:constructive_conclusion_2}
\end{align}
\end{itemize}
Combining \eqref{e:constructive_conclusion_1} and \eqref{e:constructive_conclusion_2}, we obtain  \eqref{e:energy_dissipation_condition} with a dimensional constant $C_{\text{\sc ed}}$ and an exponent $p=d+1$. This concludes our second proof of Theorem \ref{t:epi:sing}.


\section{Constructive approach via gradient flow. A third proof of Theorem \ref{t:epi:sing}}\label{s:third_proof}
In this section we review the proof of the log-epiperimetric inequality from \cite{cospve3} in terms of Theorem \ref{p:epik}.
Let $c:\partial B_1\to\R$ be as Theorem \ref{t:epi:sing}. As in the previous section, we will apply Theorem \ref{p:epik} with $\alpha=2$, $\mathcal G=W$, $\mathcal F$ as in \eqref{e:mathcalF}, and $\mathcal S$ given by \eqref{e:critical_points}. 
\medskip

\subsection{Definition of the flow} 
As in \cite{cospve3}, we define 
$$\psi\in H^1\big(]0,+\infty[\,;L^2(B_1)\big) \cap L^2\big(]0,+\infty[\,; H^2(B_1)\cap\mathcal K\big)$$
to be the strong solution of the following parabolic variational inequality (for the existence we refer to \cite{brezis})
\begin{equation}\label{e:intro:para}
\begin{cases}
\big(\psi'(t)+\nabla \mathcal F (\psi(t))\big)\cdot (v-\psi(t))\ge0\,,\quad\text{for every}\quad v\in \mathcal K\quad\text{and}\quad t>0\,,\\
\psi(0)=c\,,
\end{cases} 
\end{equation}
where $\mathcal K$ is the convex set 
$$\mathcal K=\big\{v\in L^2(\partial B_1)\ :\ v\ge 0\quad\text{on}\quad \partial B_1\big\}.$$

\subsection{Energy dissipation inequality} The energy dissipation inequality \eqref{e:energy_dissipation_condition} is automaticaly satisfied along the flow with $p=2$. Precisely, we have
\begin{equation}\label{e:para:ost:u'}
\ds \|\psi'(t)\|_{L^2(\partial B_1)}^2= - \psi'(t)\cdot \nabla \mathcal F(\psi(t))\quad\text{for almost every}\quad t>0.
\end{equation}
Indeed, by taking the test function $\psi:=u(t+h)$, for some $t>0$ and $h\in\R$, we get 
 \begin{align*}
 0\le \big(\psi(t+h)-\psi(t)\big)\cdot \big(\psi'(t)+\nabla \mathcal F(\psi(t))\big),
 \end{align*}
 Dividing by $h$ and taking the limits as $h\to 0^+$ and $h\to0^-$, we obtain the inequalities
 \begin{align*}
 0&\le \lim_{h\to0^+} \frac1h\big(\psi(t+h)-\psi(t)\big)\cdot \big(\psi'(t)+\nabla \mathcal F(\psi(t))\big)=\|\psi'(t)\|_{L^2(\partial B_1)}^2+\psi'(t)\cdot\nabla \mathcal F(\psi(t)),\\
 0&\ge \lim_{h\to0^-} \frac1h\big(\psi(t+h)-\psi(t)\big)\cdot \big(\psi'(t)+\nabla \mathcal F(\psi(t))\big)=\|\psi'(t)\|_{L^2(\partial B_1)}^2+\psi'(t)\cdot\nabla \mathcal F(\psi(t)),
 \end{align*}
 which give precisely \eqref{e:para:ost:u'}.
 
 \subsection{\L ojasiewicz inequality} 
 Now, in order to conclude the proof of the log-epiperimetric inequality (Theorem \ref{t:epi:sing}), it is sufficient to check that \eqref{e:lojaK} holds along the flow. We fix $t>0$ and we reason precisely as in \cite{cospve3}. We decompose the function $\psi(t)$ as 
 $$\psi(t)=Q+\eta_++\eta_0+\eta_-\ ,$$
 exactly as in \eqref{e:decomposition_of_c} with $\psi(t)$ in place of $c$; moreover, we define $h_2$ and $h_\alpha$ as in \eqref{e:h2halpha}, so we have 
 $$\psi(t)=h_2+h_\alpha\ .$$
 Now, by Lemma \ref{p:key}, we have that $h_2\in \mathcal K$. Thus, using \eqref{e:intro:para}, we can compute 
 $$\|\psi(t)\|_{L^2(\partial B_1)}\ge \frac{(h_2-\psi(t))\cdot\psi'(t)}{\|h_2-\psi(t)\|_{L^2(\partial B_1)}}\ge \frac{-(h_2-\psi(t))\cdot\nabla \mathcal F\big(\psi(t)\big)}{\|h_2-\psi(t)\|_{L^2(\partial B_1)}},$$
 in order to estimate the right-hand side from below, we use Lemma \ref{p:key}. 
 \begin{align*}
 \frac{-(h_2-\psi(t))\cdot\nabla \mathcal F\big(\psi(t)\big)}{\|h_2-\psi(t)\|_{L^2(\partial B_1)}}&= \frac{-h_\alpha\cdot\nabla \mathcal F(h_2+h_\alpha)}{\|h_\alpha\|_{L^2(\partial B_1)}}=\frac{2}{\|h_\alpha\|_{L^2(\partial B_1)}}\int_{\partial B_1}\Big(|\nabla_\theta \eta_+|^2-2d\,\eta_+^2\Big)\,d\theta\\
 &\ge 2\left(\|\eta_+\|_{L^2(\partial B_1)}^2+C_dM^2\right)^{-\sfrac12}\int_{\partial B_1}\Big(|\nabla_\theta \eta_+|^2-2d\,\eta_+^2\Big)\,d\theta\\
  &\ge 2\left(\|\eta_+\|_{L^2(\partial B_1)}^2+C_d\|\eta_+\|_{L^2(\partial B_1)}^{\frac{4}{d+1}}\right)^{-\sfrac12}\int_{\partial B_1}\Big(|\nabla_\theta \eta_+|^2-2d\,\eta_+^2\Big)\,d\theta\\
    &\ge C_d\|\eta_+\|_{L^2(\partial B_1)}^{-\frac{2}{d+1}}\int_{\partial B_1}\Big(|\nabla_\theta \eta_+|^2-2d\,\eta_+^2\Big)\,d\theta\ ,
 \end{align*}
 where in the last inequality we used that
 $$\|\eta_+\|_{L^2(\partial B_1)}\le \text{dist}_{2}\big(\psi(t),\mathcal S\big)\le 1,$$
 which holds for every $t\in[0,T_{\text{max}}]$, by choosing $T_{\text{max}}$ small enough and $\psi(0)$ close enough to $\mathcal S$, as in Lemma \ref{l:continuity} below.
As a consequence, we get that 
 \begin{align*}
 -\psi'(t)\cdot\nabla\mathcal F(\psi(t))=\|\psi'(t)\|_{L^2(\partial B_1)}^2&\ge C_d\|\eta_+\|_{L^2(\partial B_1)}^{-\frac{4}{d+1}}\left(\int_{\partial B_1}\Big(|\nabla_\theta \eta_+|^2-2d\,\eta_+^2\Big)\,d\theta\right)^2\\
 &\ge C_d\left(\int_{\partial B_1}\Big(|\nabla_\theta \eta_+|^2-2d\,\eta_+^2\Big)\,d\theta\right)^{\frac{2d}{d+1}}\\
  &\ge C_d\big(\mathcal F(\psi(t))-\mathcal F(\mathcal S)\big)^{\frac{2d}{d+1}},
 \end{align*}
where in the last inequality we used again Lemma \ref{p:key}. Finally, this implies \eqref{e:lojaK} with $\ds\gamma =\frac{d-1}{d+1}$. This concludes our third proof of Theorem \ref{t:epi:sing}.\qed
\medskip

We notice that the estimates from Lemma \ref{p:key} were crucial in the three proofs (section \ref{s:direct}, \ref{s:constructive} and \ref{s:third_proof}). In the first two proofs it was immediate to notice that the trace satisfies the hypotheses of Lemma \ref{p:key}. In the case of the proof that we presented in this section, we can apply Lemma \ref{p:key} because the flow $\psi$ remains close to the critical set $\mathcal S$. This follows by a standard argument that we sketch in the lemma below. 
\begin{lemma}\label{l:continuity}
For every $\eps>0$, there are constants $\delta>0$ and $T>0$ such that the following holds. If $\psi$ is a solution to \eqref{e:intro:para} and is such that 
$\text{dist}_{2}\big(\psi(0),\mathcal S\big)<\delta,$
then 	
$$\text{dist}_{2}\big(\psi(t),\mathcal S\big)<\eps\qquad\text{for every}\qquad t\in[0,T].$$
\end{lemma}	
\begin{proof}
Let $Q\in \mathcal S$ be the projection of $\psi(0)$ on $\mathcal S$, with respect to the distance $L^2(\partial B_1)$. By definition $Q$ is a critical point for $\mathcal F$ and $Q\ge 0$. Thus, using \eqref{e:intro:para}, we get 
\begin{align*}
\frac{\partial}{\partial t}\|\psi(t)-Q\|_{L^2(\partial B_1)}^2&=-2\psi'(t)\cdot\big(Q-\psi(t)\big)\le 2\big(Q-\psi(t)\big)\cdot \nabla\mathcal F(\psi(t))\\
&=-2\big(Q-\psi(t)\big)\cdot \nabla\mathcal F\big(Q-\psi(t)\big)\\
&\le 8d\int_{\partial B_1}\big(Q-\psi(t)\big)^2\,d\theta-2\int_{\partial B_1}\big(Q-\psi(t)\big)\,d\theta\\
&\le (8d+1)\int_{\partial B_1}\big(Q-\psi(t)\big)^2\,d\theta+\mathcal H^{d-1}(\partial B_1).
\end{align*} 
Now, setting $a=8d+1$ and $b=\mathcal H^{d-1}(\partial B_1)$ and applying the Gronwall inequality, we get that
$$\|\psi(t)-Q\|_{L^2(\partial B_1)}^2\le \frac{b}{a}\big(e^{at}-1\big)+e^{at}\|\psi(0)-Q\|_{L^2(\partial B_1)}^2,$$
which gives the claim.
\end{proof}

\section{On the construction of the competitor - an outline of the main ideas}\label{s:survey}
In this section we sketch the main ideas behind the proof of the log-epiperimetric inequality for the obstacle problem (Theorem \ref{t:epi:sing}) and that led us to the two constructions from \cite{cospve1}, \cite{cospve2}, \cite{cospve3}, Section \ref{s:direct} and Section \ref{s:constructive}. 
\medskip

\noindent {\bf The log-epiperimetric inequality.} We recall that given a $2$-homogeneous function $z:B_1\to\R$, in polar coordinates $z(r,\theta)=r^2c(\theta)$, our aim is to construct a competitor $h:B_1\to\R$ such that $h=z$ on $\partial B_1$ and 
\begin{equation}\label{e:log-epi-o}
W(h)-W(\mathcal S)\le W(z)-W(\mathcal S)-\eps\big|W(z)-W(\mathcal S)\big|^{1+\gamma},
\end{equation}
which can also be written as 
\begin{equation}\label{e:log-epi-oo}
W(h)-W(z)\le -\eps\big|W(z)-W(\mathcal S)\big|^{1+\gamma},
\end{equation}
where $\eps>0$, $\gamma\in[0,1)$, $\mathcal S$ is the set of singular $2$-homogeneous solutions to the obstacle problem and where we use the notation (recall that $W$ is constant on $\mathcal S$) :
$$W(\mathcal S):=W(Q)\quad\text{for every}\quad Q\in\mathcal S.$$

\subsection{The direct approach in a nutshell}\label{sub:direct} In this subsection we present the main ideas that led to the construction of the competitors in \cite{cospve1} and in Section \ref{s:direct}, and also in \cite{cospve2}, the latter in the context of the thin-obstacle problem.

\subsubsection{\bf The energy of $z$.}
We notice that if the function $h$ is such that the log-epiperimetric inequality \eqref{e:log-epi-o} holds, then it must have a lower energy than the $2$-homogeneous function $z$, so we start by analyzing the energy $W(z)$. We decompose the trace $c$ as
$$c=Q+\eta_-+\eta_0+\eta_+,$$
where $Q\in\mathcal S$, $\eta_0$ contains only lower modes, $\eta_0$ is a $(2d)$-eigenfunction of the spherical Laplacian and the Fourier expansion of $\eta_+$ contains only eigenfunctions corresponding to eigenvalues higher than $2d$. Then, we recall that 
$$W(z)-W(Q)=W_0\big(r^2\eta_-(\theta)\big)+W_0\big(r^2\eta_0(\theta)\big)+W_0\big(r^2\eta_+(\theta)\big).$$
We next examine the different terms in the right-hand side of the above identity. By Lemma \eqref{l:slicing}, we have :  
\begin{itemize}
	\item $\eta_-$ gives a negative contribution to the energy : 
	$$W_0\big(r^2\eta_-(\theta)\big)\le 0,\quad\text{the inequality being strict if}\quad  \eta_-\neq0\ ;$$ 	
	\item the energy of $\eta_0$ is zero: $W_0\big(r^2\eta_0(\theta)\big)=0$ ;  
	\item the energy of $\eta_+$ is positive: 
	$$W_0\big(r^2\eta_+(\theta)\big)\ge 0,\quad\text{the inequality being strict if}\quad  \eta_+\neq0.$$ 	
\end{itemize}	
In particular, this means that, in order to build a competitor $h$ with lower energy than $z$, we have to act on the term containing the higher modes 
$$W_0\big(r^2\eta_+(\theta)\big)=\int_{B_1}|\nabla (r^2\eta_+(\theta))|\,dx-2\int_{\partial B_1}\eta_+^2\,d\theta.$$
\subsubsection{\bf First attempt - the harmonic extension of the higher modes.} Since we are looking for a competitor that coincides with $z$ on $\partial B_1$, we cannot expect a contribution from the second (boundary) term of $W_0\big(r^2\eta_+(\theta)\big)$. 
Thus, in order to decrease the energy, one has to act on the first term, which is the Dirichlet energy of $r^2\eta_+(\theta)$. Of course, the best way to decrease the Dirichlet energy is to replace $r^2\eta_+$ by the harmonic extension of $\eta_+$ in $B_1$. Since the harmonic extension can be explicitly written in Fourier series, we get that the competitor has the form
\begin{equation}\label{e:fake}
f(r,\theta)=Q(r\theta)+r^2\eta_-(\theta)+r^2\eta_0(\theta)+\sum_{j:\alpha_j>2} c_jr^{2+\eps_j}\phi_j(\theta),
\end{equation}
where the coefficients $c_j$ are given by \eqref{e:coef_fou} and $\alpha_j$ are the corresponding homogeneities, related too the eigenvalues $\lambda_j$ of $\phi_j$ through the formula \eqref{e:alpha_j}. We also notice that $\eps_j>0$, for every $j$. In fact, since we take
$$\ds \sum_{j:\alpha_j>2} c_jr^{2+\eps_j}\phi_j(\theta)$$ 
to be precisely the harmonic extension of $\eta_+$, we have that $\eps_j=\alpha_j-2\ge 1$.

\subsubsection{\bf The energy of the competitor $f$.} We can compute the energy $W(f)$ by using Lemma \ref{l:WW0} and Corollary \ref{cor:orto} 
$$W(f)-W(Q)=W_0(f-Q)=W_0\big(r^2\eta_-(\theta)\big)+W_0\big(r^2\eta_0(\theta)\big)+\sum_{j:\alpha_j>2} c_j^2 W_0\big(r^{2+\eps_j}\phi_j(\theta)\big).$$
Now, using the Fourier expansion of $\eta_+$ and Corollary \ref{cor:orto}, we have 
\begin{equation}\label{e:eta+energy}
W_0\big(r^2\eta_+(\theta)\big)=\sum_{j:\alpha_j>2} c_j^2 W_0\big(r^{2}\phi_j(\theta)\big),
\end{equation}
so, the energy gain is given by:
\begin{equation}\label{e:estimate_fake}
W(f)-W(z)=\sum_{j:\alpha_j>2} c_j^2 \Big(W_0\big(r^{2+\eps_j}\phi_j(\theta)\big)-W_0\big(r^{2}\phi_j(\theta)\big)\Big).
\end{equation}	
\subsubsection{\bf Computation of the energy gain $W(z)-W(f)$.} In order to estimate $W(f)-W(z)$, we compute each of the terms in the right-hand side of \eqref{e:estimate_fake}. We use the fact that $\phi_j$ is an eigenfunction :
$$\int_{\partial B_1}|\nabla_\theta\phi_j|^2\,d\theta=\lambda_j \int_{\partial B_1}\phi_j^2\,d\theta=\lambda_j,$$ 
and we apply the identity \eqref{e:slicing_homo} from the Slicing Lemma \ref{l:slicing} 
\begin{align*}
W_0\big(r^{2+\eps_j} \phi_j(\theta)\big)-W_0\big(r^{2}\phi_j(\theta)\big)&= \frac{\lambda_j-2d}{d+2+2\eps_j}+\frac{\eps_j^2}{d+2+2\eps_j}-\frac{\lambda_j-2d}{d+2}\\
&=-\eps_j\frac{2(\lambda_j-2d)}{(d+2)(d+2+2\eps_j)}+\frac{\eps_j^2}{d+2+2\eps_j}.
\end{align*}
Now, a direct computation gives that if we replace $\eps_j=\alpha_j-2$ and $\lambda_j=\alpha_j(\alpha_j+d-2)$, we get  
\begin{align*}
W_0\big(r^{2+\eps_j} \phi_j(\theta)\big)-&W_0\big(r^{2}\phi_j(\theta)\big)
=\frac{\eps_j^2}{d+2+2\eps_j}\left(-\frac{2(d+2+\eps_j)}{d+2}+1\right)= -\frac{\eps_j^2}{d+2}\\
&=-\frac{\eps_j^2}{(2+\eps_j)(d+\eps_j)}\frac{\lambda_j-2d}{d+2}\le -\frac{1}{3(d+1)}\frac{\lambda_j-2d}{d+2}=-\frac{1}{3(d+1)}W_0\big(r^{2}\phi_j(\theta)\big).
\end{align*}
\subsubsection {\bf Epiperimetric inequality for the competitor $f$.} We now notice that the above estimate implies \eqref{e:log-epi-oo} in its strongest form (with $\gamma=0$): this inequality is known as {\it epiperimetric inequality}. Indeed, as a consequence of the above estimate, \eqref{e:eta+energy} and \eqref{e:estimate_fake}, we have 
$$W(f)-W(z)\le -\frac{1}{3(d+1)}\,\sum_{j:\alpha_j>2} c_j^2\,W_0\big(r^{2}\phi_j(\theta)\big) =-\frac{1}{3(d+1)}W_0\big(r^{2}\eta_+(\theta)\big).$$
Now, since the energy of $z$ is given by
\begin{equation}\label{e:z-energy-0056}
W(z)-W(Q)=W_0\big(r^{2}\eta_-(\theta)\big)+W_0\big(r^{2}\eta_+(\theta)\big)\le W_0\big(r^{2}\eta_+(\theta)\big),
\end{equation}
we get the following estimate (called epiperimetric inequality)
\begin{equation}\label{e:epi-00}
W(f)-W(z)\le -\frac{1}{3(d+1)}\big(W(z)-W(Q)\big),
\end{equation}
which is precisely \eqref{e:log-epi-oo} with $\gamma=0$, which is also the best possible exponent that we can expect. 

%

\subsubsection{\bf Why we do not use $f$ as a competitor ?} Unfortunately, the function $f$ cannot be used as a competitor in Theorem \ref{t:epi:sing} as it might not fulfill the requirement that 
\begin{center}
\rm the competitor should be non-negative.\rm	
\end{center}	
In fact, by taking the harmonic extension of $\eta_+$ (which might change sign on $\partial B_1$) we lose any information on the sign of $f$ as each of the terms $c_j\phi_j(\theta)$ of the Fourier expansion of $\eta_+$ is multiplied by a different homogeneity $r^{\alpha_j}$. 

\begin{center}
Thus, the challenge is to find a competitor that at the same time 

remains positive and decreases the energy. 

\end{center}


\subsubsection{\bf The new competitor $\widetilde f$.} 
We now try to modify the function $f$ from \eqref{e:fake} in order to have some more control on its sign, but we also try to keep the energy gain provided by the 'harmonic' competitor $f$.
The starting point is the following observation.
\begin{center}
\bf Claim. \it In \eqref{e:fake} we can take all exponents $\eps_j$ to be the same\\ and still have the epiperimetric inequality \eqref{e:epi-00}.\rm 
\end{center}
Precisely, taking in \eqref{e:fake} $\eps_j=\eps$, for every $j$, we consider the new competitor
\begin{equation}\label{e:fake2}
\widetilde f(r,\theta)=Q(r\theta)+r^2\eta_-(\theta)+r^2\eta_0(\theta)+r^{2+\eps}\eta_+(\theta).
\end{equation}

\subsubsection{\bf Estimating the energy gain ($\widetilde f$ is as good as $f$).} By using the computations that we already performed in the estimate of $W(f) - W(z)$, we can compute  
\begin{align*}
W(\widetilde f) - W(z)&=\sum_{j:\alpha_j>2} c_j^2 \Big(W_0\big(r^{2+\eps}\phi_j(\theta)\big)-W_0\big(r^{2}\phi_j(\theta)\big)\Big)\\
&= \sum_{j:\alpha_j>2} c_j^2\left(-\frac{2\eps(\lambda_j-2d)}{(d+2)(d+2+2\eps)}+\frac{\eps^2}{d+2+2\eps}\right).
\end{align*}
Now, for $\eps$ small enough the first (negative) term of the right-hand side wins against the second (positive) one. Thus, choosing $\eps$ small enough and isolating a dimensional constant $C_d$, we get 
\begin{align*}
W(\widetilde f)-W(z)&\le -\eps\, C_d\sum_{j:\alpha_j>2} c_j^2\frac{\lambda_j-2d}{d+2}=-\eps C_d\,W_0\big(r^2\eta_+(\theta)\big),
\end{align*}
which implies (after \eqref{e:z-energy-0056}) that {\bf the epiperimetric inequality holds for $\widetilde f$} :
$$ W(\widetilde f)-W(z)\le -\eps C_d\big(W(z)-W(Q)\big).$$
\subsubsection{\bf What can go wrong with $\widetilde f$ ?} Unfortunately, we still cannot prove that $\widetilde f$ is non-negative. For instance, what can go wrong is that, for some $\theta\in\partial B_1$, we have 
$$Q(\theta)+\eta_-(\theta)+\eta_0(\theta)<0\qquad\text{and}\qquad \eta_+(\theta)>-\big(Q(\theta)+\eta_-(\theta)+\eta_0(\theta)\big).$$
In this way the trace 
$$c(\theta)=Q(\theta)+\eta_-(\theta)+\eta_0(\theta)+\eta_+(\theta)$$ 
is non-negative, but the competitor $\widetilde f(r,\theta)$ becomes negative when $r$ is small enough. 

\subsubsection{\bf Construction of a positive competitor} In this section, we finally discuss the idea behind the direct constructions from  \cite{cospve1}, \cite{cospve2}, and Section \ref{s:direct}.
Precisely, in order to build a nonnegative competitor, we add a correction term $H:\partial B_1\to\R$ such that 
\begin{equation}\label{e:condition_competitor}
Q(\theta)+\eta_-(\theta)+\eta_0(\theta)+H(\theta)\ge 0\qquad\text{for every}\qquad \theta \in\partial B_1,
\end{equation}
and we consider the competitor
$$h(r,\theta)=r^2\big(Q(\theta)+\eta_-(\theta)+\eta_0(\theta)+H(\theta)\big)+r^{2+\eps}\big(\eta_+(\theta)-H(\theta)\big).$$
Since, by hypothesis the trace $c=Q+\eta_-+\eta_0+\eta_+$ is non-negative, we get that 
$$\big(Q+\eta_-+\eta_0+H\big)+\big(\eta_+-H\big)\ge 0,$$
but (together with \eqref{e:condition_competitor}) this implies that $h(r,\theta)\ge 0$ for every $r>0$.

\subsubsection{\bf The challenge behind the choice of $H$.} Notice that, if we want the condition \eqref{e:condition_competitor} to be fulfilled, we need $H$ to be large enough in order to compensate the negative part of $Q+\eta_-+\eta_0$. On the other hand,  $H$ increases the energy. In fact, we can re-write the competitor $h$ as 
$$h(r,\theta)=\widetilde f(r,\theta)+\big(r^2-r^{2+\eps}\big)H(\theta).$$
Since $\widetilde f$ is qualitatively the best possible choice for the energy (recall that $\widetilde f$ is as good as the harmonic extension $f$), the function $h$ will have bigger energy, which, of course, depends on the error introduced by the correction term $\big(r^2-r^{2+\eps}\big)H(\theta)$; finally, this means that in order to keep the energy of $h$ as small as possible, we need $H$ to be small. 

\begin{center}
This competition between the  constraint and the energy is precisely what\\ makes appear the exponent $\gamma$ in the log-epiperimetric inequality. 
\end{center}
Following the construction presented here, one can build many different competitors. For instance, in \cite{cospve1}, we use a function $H$ that depends on all the lower modes (including the linear ones) of the trace $c$. In Section \ref{s:direct} we propose a different function $H$, which is $(2d)$-eigenfunction on the sphere; this leads to a shorter proof, but the exponent $\gamma$ we get is not optimal.


\subsection{Constructive approach via a gradient flow}\label{sub:constructive} This section is dedicated to the constructive approach from \cite{cospve3} and Section \ref{s:constructive}. It is based on the idea that any function $h:B_1\to\R$ can be seen as a family of functions (a flow) $h(r,\cdot):\partial B_1\to\R$ parametrized over the radial coordinate $r\in(0,1]$. This way to see the competitor was first used in the context of the one-phase Bernoulli problem, in \cite{spve} and later in \cite{esv1}, where the competitor is not explicit, but is constructed starting from a solution of an evolution problem. Recently, in \cite{cospve3} we applied this idea to the case of the obstacle and the thin-obstacle problems. In Section \ref{s:constructive}, we used a general result (that we prove in the Appendix) and we construct a new flow, which simulates the behavior of the gradient flow from \cite{cospve3}, but is also closely related to the explicit competitor  from Section \ref{s:constructive}. As in the previous Section \ref{sub:direct}, we proceed by dividing the exposition in several paragraphs, each one representing a different step of the construction.

\medskip


\subsubsection{\bf Slicing Lemma} The starting point is the slicing lemma (Lemma \ref{l:slicing}) which allows to write down the energy of the competitor $h(r,\theta)$ as an integral over the different spheres (slices) $\partial B_r$, $r\in(0,1]$. Precisely, one can compute that (see Lemma \ref{l:slicing}) if $h$ is of the form 
$$h(r,\theta)=r^2u(r,\theta),$$ 
then its energy $W(h)$ is given by
\begin{equation}\label{e:slicing-oo}
W(h)=W(r^2u)=\int_0^1 \mathcal F\big(u(r,\cdot)\big) r^{d+1} \,dr+\int_0^1r^{d+3}\int_{\partial B_1} |\de_r u|^2\,d\HH^{d-1}\,dr,
\end{equation}
where $\mathcal F$ is a functional acting on $H^1(\partial B_1)$. 

Thus, we will search for a competitor of the form  $h(r,\theta)=r^2u(r,\theta)$, where $u$ can be read as a one-parameter family of functions 
$$u(r,\cdot):\partial B_1\to\R\,,\quad r\in(0,1].$$ 


\subsubsection{\bf Energy of the $2$-homogeneous extension $z$.} In this framework, the $2$-homogeneous extension $z$, given in polar coordinates by
$$z(r,\theta)=r^2c(\theta),$$ 
corresponds to the case in which the flow $r\mapsto u(r,\cdot)\in H^1(\partial B_1)$ is constant in $r$. In this case, the second term in the right-hand side of \eqref{e:slicing-oo} is zero and so we get
$$W(z)=\int_0^1 \mathcal F\big(c\big) r^{d+1} \,dr=\frac1{d+2}\mathcal F(c).$$

\subsubsection{\bf The log-epiperimetric inequality in terms of $\mathcal F$.} As a consequence, the log-epiperimetric inequality \eqref{e:log-epi-o}, for $h(r,\theta)=r^2 u(r,\theta)$ and the $2$-homogeneous extension $z(r,\theta)=r^2c(\theta)$, can be written in terms of the new functional $\mathcal F$ as: 
\begin{equation}\label{e:log-epi-ooo}
\int_0^1 \Big(\mathcal F\big(u(r,\cdot)\big)-\mathcal F(c)\Big) r^{d+1} \,dr+\int_0^1r^{d+3}\int_{\partial B_1} |\de_r u|^2\,d\HH^{d-1}\,dr\le -\eps\big(\mathcal F(c)-\mathcal F(\mathcal S)\big)^{1+\gamma},
\end{equation}
where $\mathcal F(S):=(d+2)W(\mathcal S)$.
Thus, in order to find a function $u$ for which \eqref{e:log-epi-ooo} holds, we have to search for a function $u:(0,1]\times \partial B_1\to\R$, which in particular satisfies 
$$u(1,\theta)=c(\theta)\quad\text{for}\quad \theta\in\partial B_1\qquad\text{and}\qquad \mathcal F(u(r,\cdot))\le \mathcal F(c)\quad\text{for}\quad r\le 1,$$
but we also have to take into account the cost of modifying $u$ from one scale to another, which is given by the error term 
$$\int_0^1r^{d+3}\int_{\partial B_1} |\de_r u|^2\,d\HH^{d-1}\,dr.$$

\subsubsection{\bf Reparametrization of the flow.} Now, since we will work with flows, it is convenient to consider functions of the form $\varphi:[0,+\infty)\to H^1(\partial B_1)$ and to define the competitor $u$ as 
$$u(r,\theta)=\varphi(- \kappa\ln r,\theta)\quad\text{for}\quad r\in(0,1]\quad\text{and}\quad \theta\in\partial B_1.$$
Thus, we have that 
$$\varphi(0,\cdot)=c(\cdot)\quad\text{on}\quad \partial B_1,$$
and the energy of the competitor 
$$h(r,\theta)=r^2\varphi(- \kappa\ln r,\theta),$$
can be written as (see Lemma \ref{l:slicing2})
$$W(h)=\frac{\mathcal F(\varphi(0))}{d+2}+\int_0^{+\infty} \left(\frac{1}{d+2} \, \varphi'(t)\cdot\nabla\mathcal F(\varphi(t))+ \kappa\|\varphi'(t)\|^{2}_{L^2(\partial B_1)}  \right)\,e^{-\frac{t(d+2)}\kappa}   \,dt,$$
or, alternatively, as 
$$W(h)=\frac{\mathcal F(\varphi(0))}{d+2}+\frac1\kappa\int_0^\infty \Big(\mathcal F(\varphi(t)) -\mathcal F(\varphi(0))\Big)\, e^{-\frac{t(d+2)}\kappa}\,dt+\kappa\int_{0}^\infty\|\varphi'(t)\|_{L^2(\partial B_1)}^2   \,e^{-\frac{t(d+2)}\kappa}\,dt.$$

\subsubsection{\bf The best flow is the gradient flow.} The last expression of the energy $W(h)$ suggests that, in order to construct a competitor with lower energy, we have to act on the term $\mathcal F(\varphi(t))$. A natural way to do so, is to choose $\varphi=\psi$, where $\psi$ is the gradient flow of $\mathcal F$
\begin{equation}\label{e:heat}
\begin{cases}\psi'(t)=-\nabla\mathcal F(\psi(t))\quad\text{for}\quad t\in(0,+\infty),\\
\psi(0)=c.\end{cases}
\end{equation}

\subsubsection{\bf The role of the stopping time.}  The basic idea behind any of the constructions we propose is, given a homogeneous function $z:B_1\to\R$, to build a competitor that simulates the behavior of the solution (in our case, the solution of the obstacle problem) with the same values as $z$ on $\partial B_1$. Now, for the obstacle problem, it is well known that if $h(r,\theta)=r^2u(r,\theta)$ is precisely the solution of the obstacle problem, then the energy $\mathcal F(u(r,\cdot))$ remains positive for every $r>0$ (this is a consequence of the Weiss' monotonicity formula \cite{weiss}). Thus, we do not expect to find a good competitor with negative energy. So, instead of taking $\varphi$ to be precisely the gradient flow $\psi$, we stop $\psi$ at a certain time $T$. The choice of the stopping time is fundamental step in the epiperimetric inequality and will be chosen in function of the energy $\mathcal F(\psi(T))$.

\subsubsection{\bf Definition of the flow and the competitor.}\label{subs:competitor} We define the flow $\varphi(t,\theta)$ as 
$$\varphi(t):=\psi(t)\quad\text{if}\quad t\in[0,T],\qquad \varphi(t):=\psi(T)\quad\text{if}\quad t\ge T,$$
and teh competitor $h$ as 
$$h(r,\theta)=r^2\varphi(-\kappa \ln r,\theta),$$
where $\kappa$ and $T$ will be chosen below. 

\subsubsection{\bf The energy of the competitor $h$.} Again, the energy of the competitor $u$ can be expressed in two different ways (we refer to Lemma \ref{l:slicing2} for the computation): 
\begin{equation}\label{e:first_en_eq}
W(h)-W(z)=\int_0^T \left(\frac1{d+2}\,\psi'(t)\cdot\nabla \mathcal F(\psi(t)) + \kappa\|\nabla \mathcal F(\psi(t))\|_{L^2(\partial B_1)}^2\right)\,e^{-\frac{t(d+2)}\kappa}\,dt,
\end{equation}
where we recall that $W(z)=\frac1{d+2}\mathcal F(\psi(0))$, and 
\begin{align}\label{e:second_en_eq}
W(h)-W(z)&=\frac1\kappa\int_0^T \Big(\mathcal F(\psi(t)) -\mathcal F(\psi(0))\Big)\, e^{-\frac{t(d+2)}\kappa}\,dt+\frac1{d+2} \Big(\mathcal F(\psi(T)) -\mathcal F(\psi(0))\Big)\, e^{-\frac{T(d+2)}\kappa}\notag\\
&\qquad\qquad \qquad\qquad\qquad+\kappa\int_{0}^T\|\psi'(t)\|_{L^2(\partial B_1)}^2   \,e^{-\frac{t(d+2)}\kappa}\,dt.
\end{align}
Moreover, we notice that, by summing up \eqref{e:first_en_eq} and \eqref{e:second_en_eq}, we obtain 
\begin{align*}
W(h)-W(z)&=\frac1{2(d+2)}\int_0^T \psi'(t)\cdot\nabla \mathcal F(\psi(t))\,e^{-\frac{t(d+2)}\kappa}\,dt\notag\\
&\qquad\qquad +\frac{1}{2\kappa}\int_0^T \Big(\mathcal F(\psi(t)) -\mathcal F(\psi(0))\Big)\, e^{-\frac{t(d+2)}\kappa}\,dt\notag\\
&\qquad\qquad \qquad\qquad +\frac1{2(d+2)} \Big(\mathcal F(\psi(T)) -\mathcal F(\psi(0))\Big)\, e^{-\frac{T(d+2)}\kappa}\notag\\
&\qquad\qquad \qquad\qquad\qquad\qquad +\kappa\int_{0}^T\|\psi'(t)\|_{L^2(\partial B_1)}^2   \,e^{-\frac{t(d+2)}\kappa}\,dt.\notag
\end{align*}

\subsubsection{\bf Estimating the energy gain $W(h)-W(z)$.} First, notice that the second term in the right-hand side of the above identity is negative. This is due to the fact that the energy $\mathcal F$ is decreasing along the flow $\psi(t)$. Thus, we get 
\begin{align}
W(r^2u)-W(z)&\le \frac1{2(d+2)}\int_0^T \psi'(t)\cdot\nabla \mathcal F(\psi(t))\,e^{-\frac{t(d+2)}\kappa}\,dt\label{e:main_en_eq_1}\\
&\qquad\qquad  +\frac1{2(d+2)} \Big(\mathcal F(\psi(T)) -\mathcal F(\psi(0))\Big)\, e^{-\frac{T(d+2)}\kappa}\label{e:main_en_eq_2}\\
&\qquad\qquad \qquad\qquad+\kappa\int_{0}^T\|\psi'(t)\|_{L^2(\partial B_1)}^2   \,e^{-\frac{t(d+2)}\kappa}\,dt.\label{e:main_en_eq_3}
\end{align}
Now since for the gradient flow $\psi$ we have the identities
\begin{equation}\label{e:heat_derivative}
\frac{\partial}{\partial t}\mathcal F(\psi(t))=\psi'(t)\cdot\nabla \mathcal F(\psi(t))=-\|\nabla \mathcal F(\psi(t))\|_{L^2(\partial B_1)}^2=-\|\psi'(t)\|_{L^2(\partial B_1)}^2,
\end{equation}
the terms \eqref{e:main_en_eq_1} and \eqref{e:main_en_eq_2} are negative. Thus, we start by estimating the third term \eqref{e:main_en_eq_3}. 
\subsubsection{\bf Eliminating the error term \eqref{e:main_en_eq_3} - the choice of $\kappa$.} Using again \eqref{e:heat_derivative}, we get that \eqref{e:main_en_eq_3} can be absorbed in \eqref{e:main_en_eq_1}. Indeed, by choosing $\kappa$ small enough, for instance,
$$\kappa= \frac1{4(d+2)},$$ 
we obtain
\begin{align}
W(h)-W(z)&\le \frac1{4(d+2)}\int_0^T \psi'(t)\cdot\nabla \mathcal F(\psi(t))\,e^{-\frac{t(d+2)}\kappa}\,dt\label{e:1main_en_eq_1}\\
&\qquad\qquad  +\frac1{2(d+2)} \Big(\mathcal F(\psi(T)) -\mathcal F(\psi(0))\Big)\, e^{-\frac{T(d+2)}\kappa}.\label{e:1main_en_eq_2}
\end{align}
Now, both terms \eqref{e:1main_en_eq_1} and \eqref{e:1main_en_eq_2} are negative.

\subsubsection{\bf Choice of the stopping time $T$.} 
Let $T_{\sfrac12}$ be defined as
$$T_{\sfrac12}=\sup\Big\{\tau\ :\ \mathcal F(\psi(t))-\mathcal F(\mathcal S)\ge \frac12\big(\mathcal F(\psi(0))-\mathcal F(\mathcal S)\big)\quad\text{for every}\quad t\in[0,\tau]\Big\}.$$
We consider two cases.

\noindent {\bf Case 1. \it The energy decreases rapidly along the flow: $T_{\sfrac12}\le 1$.} In this case, we choose $T=T_{\sfrac12}$ and we compute 
\begin{align*}
W(h)-W(z)&\le \frac{e^{-\frac{T(d+2)}\kappa}}{2(d+2)} \Big(\mathcal F(\psi(T)) -\mathcal F(\psi(0))\Big)\\
&\qquad\le \frac{e^{-\frac{d+2}\kappa}}{2(d+2)} \Big(\mathcal F(\psi(T)) -\mathcal F(\psi(0))\Big)\\
&\qquad\qquad= \frac{e^{-\frac{d+2}\kappa}}{2(d+2)} \Big(\big(\mathcal F(\psi(T)) -\mathcal F(\mathcal S)\big)-\big(\mathcal F(\psi(0))-\mathcal F(\mathcal S)\big)\Big)\\
&\qquad\qquad\qquad=-\frac12 \frac{e^{-\frac{d+2}\kappa}}{2(d+2)} \Big(\mathcal F(\psi(0))-\mathcal F(\mathcal S)\Big),
\end{align*}
which concludes the proof of the log-epiperimetric inequality \eqref{e:log-epi-o} (equivalently \eqref{e:log-epi-ooo}) in this first case, in which we get \eqref{e:log-epi-o} with the best possible exponent $\gamma=0$.
\medskip

\noindent {\bf Case 2. \it The energy decreases slowly along the flow: $T_{\sfrac12}\ge 1$.} In this case, we choose $T=1$ and the energy gain comes from the term \eqref{e:1main_en_eq_1}. Indeed, 
$$W(h)-W(z)\le \int_0^T \psi'(t)\cdot\nabla \mathcal F(\psi(t))\,e^{-\frac{t(d+2)}\kappa}\,dt= -\int_0^T \|\nabla \mathcal F(\psi(t))\|_{L^2(\partial B_1)}^2\,e^{-\frac{t(d+2)}\kappa}\,dt.$$
In this second case, the proof of \eqref{e:log-epi-ooo} is more involved and is based on the so-called \L ojasiewicz inequality.

\subsubsection{\bf \L ojasiewicz inequality for the functional $\mathcal F$.} In order to conclude the proof of \eqref{e:log-epi-ooo} also in the second case, we need to estimate $\|\nabla \mathcal F(\psi(t))\|$ from below in terms of the energy $\big(\mathcal F(\psi(t))-\mathcal F(\mathcal S)\big)$. Precisely, we need an inequality of the form 
\begin{equation}\label{e:loja0}
C_{\text{\sc ls}}\big(\mathcal F(\phi)-\mathcal F(\mathcal S)\big)^{1+\beta}\le\|\nabla \mathcal F(\phi)\|^2\qquad\text{for every}\quad \phi\quad\text{such that}\quad\mathcal F(\phi)\ge \mathcal F(\mathcal S),
\end{equation}
where $C_{\text{\sc ls}}>0$ and $0\le \beta<1$. 
The above estimate is called \L ojasiewicz inequality and is well-known in the case when $\mathcal F$ is an analytic function on a finite dimensional space; we refer to \cite{cospve3} for a more detailed discussion on the \L ojasiewicz inequality and its different versions and applications.  Now, using this estimate and the choice of the stopping time $T=1\le T_{\sfrac12}$ and $\kappa=\frac1{4(d+2)}$, we can estimate \eqref{e:1main_en_eq_1} as follows: 
\begin{align*}
-\int_0^T \|\nabla \mathcal F(\psi(t))\|_{L^2(\partial B_1)}^2\,e^{-\frac{t(d+2)}\kappa}\,dt
&\le -C_{\text{\sc ls}}\int_0^T \big(\mathcal F(\varphi(t))-\mathcal F(\mathcal S)\big)^{1+\beta}\,e^{-\frac{t(d+2)}\kappa}\,dt\\
&\le -\frac{C_{\text{\sc ls}}}{2^{1+\gamma}}\int_0^T \big(\mathcal F(\varphi(0))-\mathcal F(\mathcal S)\big)^{1+\beta}\,e^{-\frac{t(d+2)}\kappa}\,dt\\
&\le -\frac{C_{\text{\sc ls}}\kappa}{(d+2)2^{1+\beta}}\left(1-e^{-\frac{T(d+2)}\kappa}\right) \big(\mathcal F(\varphi(0))-\mathcal F(\mathcal S)\big)^{1+\beta}\\
&= -C_dC_{\text{\sc ls}}\big(\mathcal F(\varphi(0))-\mathcal F(\mathcal S)\big)^{1+\beta},
\end{align*}
where $C_d$ is a dimensional constant. This concludes the proof of the log-epiperimetric inequality \eqref{e:log-epi-ooo} in the second case; the exponent $\gamma=\beta$ is precisely the one from the \L ojasiewicz inequality.

\subsubsection{\bf What can go wrong with the gradient flow ?} Up to this point, we have proved that the competitor $h$ from Section \ref{subs:competitor} satisfies the log-epiperimetric inequality provided that the \L ojasiewicz inequality \eqref{e:loja0} holds along the flow. For what concerns the functional $\mathcal F$ the \L ojasiewicz inequality holds and is relatively easy to prove (see for instance the Introduction of \cite{cospve3}). On the other hand, in order to conclude that $h$ is an admissible competitor in Theorem \ref{t:epi:sing}, we must have that $h$ is non-negative, or in terms of the flow $\psi$, that $\psi(t)$ is non-negative on $\partial B_1$, for every $t\in[0,T]$. Unfortunately, we cannot assure that, for any $\psi(0)=c$, the flow remains positive. Thus, in \cite{cospve3}, we propose a different construction.

\subsubsection{\bf Constrained gradient flow.}
In \cite{cospve3}, we construct the competitor $h$ from Subsection \ref{subs:competitor} starting from a flow $\psi$, which is the gradient flow of $\mathcal F$ constrained to remain in the (convex) space $\mathcal K$ of nonnegative functions defined on $\partial B_1$. This constrained flow, is of course different with respect to the original gradient flow as the decay of the energy $\mathcal F$ may become much slower when the flow hits the boundary of the constraint $\mathcal K$, but still, this flow has several properties, that make it very similar to the (unconstrained) gradient flow of $\mathcal F$. In particular, we can repeat precisely the same construction presented above: the equalities \eqref{e:first_en_eq} and \eqref{e:second_en_eq} are general and hold for any function $\psi:[0,+\infty)\to H^1(\partial B_1)$, while the identities \eqref{e:heat_derivative} should be replaced by
\begin{equation}\label{e:heat_derivative_constrained}
\psi'(t)\cdot\nabla \mathcal F(\psi(t))=-\|\nabla \mathcal F(\psi(t))\|_{\mathcal K}^2=-\|\psi'(t)\|_{L^2(\partial B_1)}^2. 
\end{equation}
Thus, the proof of the epiperimetric inequality is precisely the same, with the only difference that the norm of $\nabla \mathcal F$ is replaced by  
\begin{equation}\label{e:Knorm}
\|\nabla \mathcal F(\phi)\|_{\mathcal K}=\sup\left\{0\,,\,\sup_{v\in \mathcal K\setminus\{\phi\}}\left\{ \frac{-(v-\phi)\cdot\nabla\mathcal F(\phi)}{\|v-\phi\|_{L^2(\partial B_1)}} \right\}\right\},
\end{equation}
for any nonnegative $\phi\in H^2(\partial B_1)$. Now, the positivity constraint for this flow is automatically satisfied, so the main challenge is to prove an estimate that can replace the \L ojasiewicz inequality \eqref{e:loja0}. Indeed, in order to complete the proof, in \cite{cospve3}, we prove the following stronger version of \eqref{e:loja0}, that we called {\it constrained \L ojasiewicz inequality} :
\begin{equation}\label{e:loja-constrained}
C_{\text{\sc cls}}\big(\mathcal F(\phi)-\mathcal F(\mathcal S)\big)^{1+\beta}\le\|\nabla \mathcal F(\phi)\|_{\mathcal K}^2\ .
\end{equation}
The proof is based on the choice of a suitable test direction $\phi$ in \eqref{e:Knorm}, that turns out to be precisely the function $h_2$ from \eqref{e:condition_competitor}. 


\subsubsection{\bf The log-epiperimetric inequality - a general construction.} In Theorem \ref{p:epik} of the Appendix, we make a more general construction of a competitor $h$ starting from a flow $\psi:[0,+\infty)\to H^1(\partial B_1)$ that satisfies 
satisfies the following conditions : 
\begin{enumerate}[(i)]
	\item for any $t\ge 0$, the function $\psi(t)$ remains nonnegative along the flow; this assures that the final competitor is admissible;
	\item the following inequality holds :
	$$-\psi'(t)\cdot \nabla\mathcal F(\psi(t))\ge \|\psi'(t)\|^2_{L^2(\partial B_1)}\quad\text{for every}\quad t>0\ ;$$
	this guarantees that the energy $\mathcal F(\psi(t))$ is decreasing in $t$ and that the error term \eqref{e:main_en_eq_3} can be absorbed by the energy gain \eqref{e:main_en_eq_1};
	\item the following \L ojasiewicz-type inequality (which replaces \eqref{e:loja0}) hold
	\begin{equation}\label{e:loja0}
	C_{\text{\sc lst}}\big(\mathcal F(\psi(t))-\mathcal F(\mathcal S)\big)^{1+\beta}\le-\psi'(t)\cdot \nabla \mathcal F(\psi(t))\ ,
	\end{equation}
	for every $t>0$ such that $\psi'(t)\neq 0$.
\end{enumerate}	 
This abstract result can be used also in other contexts (for instance, in can be applied to the thin-obstacle problem). In Section \ref{s:direct} we apply Theorem \ref{p:epik} to a specific flow, for which the derivative $\psi'(t)$ does not depend on $t$ and we choose the direction $\psi'$ to be precisely the one from \eqref{e:condition_competitor}; thus, recovering the competitor from Section \ref{s:direct}. 

\newpage

\appendix 

\section{Evolution approach to the log-epiperimetric inequality}


In this section we give a general procedure that reduces the construction of a competitor for the log-epiperimetric inequality to the construction of a flow that satisfies two key hypotheses: an energy dissipation estimate and a \L ojasiewicz inequality. Our construction applies not only to the specific case of the obstacle problem, but can be used to prove log-epiperimetric inequalities for any functional that satisfies suitable homogeneity properties. In particular, it can be used in the context of to the thin-obstacle problem and, more generally, to the obstacle problem for the $s$-Laplacian. 
Our main result is Theorem \ref{p:epik}. Before stating it, we introduce some notation and we list the main assumptions that we make. 
\medskip 

\noindent {\bf Homogeneity.} We fix a positive real constant $\alpha>0$; in the case of the obstacle problem (that is, in Theorem \ref{t:epi:sing}) $\alpha$  is equal to $2$. 

\medskip 

\noindent {\bf Energy.} We consider two functionals 
$$\mathcal G:H^1(B_1)\to\R\qquad\text{and}\qquad  \mathcal F: H^1(\partial B_1)\to\R\ ,$$ 
with the following properties. 
\begin{itemize}
	\item $\mathcal F$ is differentiable. Precisely, there is a functional $\nabla\mathcal F:H^2(\partial B_1)\to L^2(\partial B_1)$ such that 
	$$\mathcal F(u+v)=\mathcal F(u)+v\cdot \nabla\mathcal F(u)+o\left(\|v\|_{H^1(\partial B_1)}\right),$$
	for every $u\in H^2(\partial B_1)$ and every $v\in H^1(\partial B_1)$.\vspace*{0.1cm}
	\item $\mathcal G$ and $\mathcal F$ are related through a slicing identity. Precisely, we assume that there is a constant $C_{\text{\tiny\sc sl}}>0$ such that, for any function $u=u(r,\theta)\in H^1([0,1]\times \de B_1)$, we have 
	\begin{equation}\label{e:slicingG}
	\mathcal G\big(r^\alpha\,u(r,\theta)\big)\leq  \int_0^1 \mathcal F(u(r,\cdot))\, r^{2\alpha+d-3}\, dr+C_{\text{\tiny\sc sl}}\, \int_{0}^1\int_{\de B_1} |\de_r u|^2\,r^{2\alpha+d-1}\,d\HH^{d-1}dr\,,
	\end{equation} 
	with equality if $u$ is constant in the $r$ variable. In this case, we have 
	\begin{equation}\label{e:slicingG2}
	\mathcal G\big(r^\alpha u(\theta)\big)= \int_0^1 \mathcal F(u)\, r^{2\alpha+d-3}\, dr=\frac{1}{2\alpha+d-2}\mathcal F(u).
	\end{equation} 
\end{itemize}
\medskip

\noindent{\bf Critical set.} We suppose that there is a compact set $\mathcal S\subset H^2(\partial B_1)$ such that: 
\begin{itemize}
	\item $\mathcal S$ is a set of critical points for $\mathcal F$, that is :
	$$\nabla\mathcal F(Q)=0\quad\text{for every}\quad Q\in\mathcal S.$$
	\item $\mathcal F$ is constant on $\mathcal S$; and we denote this constant by $\mathcal F(\mathcal S)$:
	$$\mathcal F(Q)=\mathcal F(\mathcal S)\quad\text{for every}\quad Q\in\mathcal S.$$

\end{itemize}

\medskip

\noindent{\bf Flow.} We suppose to be given a constant $T_{\text{max}}>0$ and a function $\psi: [0,T_{\text{max}}]\to H^2(\partial B_1)$ such that 
$$\psi\in L^2\big([0,T_{\text{max}}]; H^2(\partial B_1)\big)\,\cap\, H^1\big((0,T_{\text{max}}); L^2(\partial B_1)\big),$$
$$or\qquad \psi\in L^2\big([0,T_{\text{max}}]; H^1(\partial B_1)\big)\,\cap\, H^1\big((0,T_{\text{max}}); H^1(\partial B_1)\big).$$
In both cases the energy $\mathcal F(\psi(t))$ is well-defined and (weakly) differentiable in $t$. Precisely, 
$$\frac{d}{dt}\mathcal F(\psi(t))=\psi'(t)\cdot \nabla \mathcal F(\psi(t))\quad \text{for every}\quad t\in(0,T_{\text{max}}),$$
and the map 
$$(0,T_{\text{max}})\ni t\mapsto \psi'(t)\cdot \nabla \mathcal F(\psi(t))\in\R$$ 
is integrable, where the dot indicates the scalar product in $L^2(\partial B_1)$, or the pairing between $H^1(\partial B_1)$ and its dual space. Moreover, we assume that the flow and the energy $\mathcal F$ satisfy the following properties. 
\begin{itemize}
	\item {\bf Energy dissipation inequality.} There are constants $C_{\text{\tiny\sc ed}}>0$ and $p\ge2$ such that the following inequality holds
	\begin{equation}\label{e:energy_dissipation_condition}
	C_{\text{\tiny\sc ed}}\min\Big\{\|\psi'(t)\|_{L^2(\partial B_1)}^{2},\|\psi'(t)\|_{L^2(\partial B_1)}^{p}\Big\}\le -\psi'(t)\cdot \nabla \mathcal F(\psi(t)),
	\end{equation}
	for almost every $t\ge 0$. In particular, the energy is non-increasing along the flow : 
	\begin{equation}\label{e:decreasing_energy}
	\mathcal F(\psi(t))-\mathcal F(\psi(s))=\int_s^t\psi'(\tau)\cdot \nabla \mathcal F(\psi(\tau))\,d\tau\le 0\qquad\text{for every}\qquad 0\le s\le t.
	\end{equation}
	\item {\bf \L ojasiewicz inequality.} There are constants $C_{\text{\tiny\sc ls}}>0$ and $\beta\in[0,1)$ such that $\mathcal F$ satisfies the following inequality along the flow
	\begin{equation}\label{e:lojaK}
	C_{\text{\tiny\sc ls}}\big(\mathcal F(\psi(t))-\mathcal F(\mathcal S)\big)^{1+\beta}\le  -\psi'(t)\cdot \nabla \mathcal F(\psi(t))\quad\text{for almost every}\quad t>0.
	\end{equation}
\end{itemize}

\begin{teo}\label{p:epik}
	Suppose that the functionals $\mathcal G$ and $\mathcal F$, and the flow $\psi$ satisfy the hypotheses above. Moreover, we assume that the exponents $\beta\in[0,1)$ and $p\ge 2$, from \eqref{e:energy_dissipation_condition} and \eqref{e:lojaK} respectively, are such that 
	$$(1+\beta)\left(1-\frac1p\right)<1.$$ 
	Then, there are constants $\delta_0>0$, $E>0$, $\gamma\in[0,1)$ and $\eps>0$, depending on $d$, $\alpha$, $p$, $\beta$, $T_{\text{max}}$, $C_{\text{\sc sl}}$, $C_{\text{\sc ls}}$ and $C_{\text{\sc ed}}$, such that the following holds. If $c\in H^1(\partial B_1)$ satisfies  
	$$c=\psi(0)
	\qquad\text{and}\qquad \mathcal F(c)-\mathcal F(\mathcal S)\le E,$$ 
	then there exists a function $h=h(r,\theta)\in H^1(B_1)$ satisfying $h(1,\cdot)=c(\cdot)$ on $\partial B_1$, and 
	\begin{equation}\label{e:logepiK}
	\mathcal G(h)-\mathcal G(\mathcal S)\leq \big(1-\eps |\mathcal G(z)-\mathcal G(\mathcal S)|^{\gamma} \big) \big(\mathcal G(z)-\mathcal G(\mathcal S)\big)
	\end{equation}
	where $\gamma=(1+\beta)\big(2-\frac2p\big)-1$, and where we used the notations
	$$z(r,\theta):=r^\alpha c(\theta)\quad\text{and}\quad \mathcal G(\mathcal S):=\frac{1}{2\alpha+d-2} \mathcal F(\mathcal S).$$
\end{teo}

\begin{oss}[The two extremal cases]
	When $p=2$ and $\beta>0$, Theorem \ref{p:epik} corresponds precisely to \cite[Proposition 3.1]{cospve3}. On the other hand, in Section \ref{s:constructive}, we apply Theorem \ref{p:epik} to a flow for which $p>2$ and $\beta=0$.
\end{oss}

\begin{oss}[About a missing hypothesis]
	In the proposition above there is one hypothesis less with respect to \cite[Proposition 3.1]{cospve3} and Theorem \ref{t:epi:sing}, where it is also required that the trace $c$ is $L^2(\partial B_1)$-close to the set $\mathcal S$ of critical points. This closeness condition is hidden in the hypotheses that the energy dissipation and the \L ojasiewicz inequalities \eqref{e:energy_dissipation_condition} and \eqref{e:lojaK} hold for every $t$ along the flow $\psi$. In fact, in Section \ref{s:constructive}, in order to prove \eqref{e:energy_dissipation_condition} and \eqref{e:lojaK} are satisfied for our specific choice of the flow, we use the closeness condition, which was essential in the proof of the key estimates in Section \ref{s:best}; similarly, in \cite[Proposition 3.1]{cospve3}, we used that the trace $c$ lies close to $\mathcal S$ in the proof of the \L ojasiewicz inequality.
\end{oss}	

\begin{oss}[About $\mathcal G(\mathcal S)$] Let $Q\in \mathcal S$. Then, $\mathcal G(\mathcal S)$ is precisely the energy $\mathcal G(r^\alpha Q(\theta))$ of the $\alpha$-homogeneous extension $r^\alpha Q(\theta)$ of $Q$.
\end{oss}

	\noindent {\bf Proof of Theorem \ref{p:epik}.} First, notice that if $\mathcal G(z)-\mathcal G(\mathcal S)\leq 0$, then choosing $h=z$ we immediately get \eqref{e:logepiK}. Throughout the rest of the proof we will assume that 
	$$0<\mathcal G(z)-\mathcal G(\mathcal S)=\frac{1}{2\alpha+d-2}\big(\mathcal F(c)-\mathcal F(\mathcal S)\big).$$
	We define the competitor $h$ as
	$$h(r,\theta)=r^\alpha u(r,\theta),$$
	where, as in Section \ref{sub:constructive}, 
	$$u(r,\theta)=\varphi(- \kappa\ln r,\theta)\quad\text{for}\quad r\in(0,1]\quad\text{and}\quad \theta\in\partial B_1,$$
	for some $\kappa>0$, and $\varphi$ is the stopped flow
	$$\varphi(t):=\psi(t)\quad\text{if}\quad t\in[0,T],\qquad \varphi(t):=\psi(T)\quad\text{if}\quad t\ge T,$$
	where the stopping time $T$ will be chosen later. 
	\medskip
	
	We will divide the rest of the proof in several steps. Before we proceed, we notice that the log-epiperimetric inequality for $\mathcal G$ \eqref{e:logepiK} is equivalent to: 
	\begin{equation}\label{e:logepiK2}
	\mathcal G(h)-\mathcal G(z)\le -\eps \big(\mathcal G(z)-\mathcal G(\mathcal S)\big)^{1+\gamma},
	\end{equation}
	where the right-hand side of the above inequality can also be written as
	$$\eps \big(\mathcal G(z)-\mathcal G(\mathcal S)\big)^{1+\gamma}=\frac{\eps}{(2\alpha+d-2)^{1+\gamma}}\big(\mathcal F(c)-\mathcal \mathcal F(\mathcal S)\big)^{1+\gamma}.$$
	We start with estimating from above the energy gap $\mathcal G(h)-\mathcal G(z)$ in terms of the flow $\mathcal \psi$.
	\medskip
	
	\noindent {\bf Estimating the energy gap.} We first give the energy $\mathcal G(h)$ in terms of the flow $\psi$ and the variable $t$. Using \eqref{e:slicingG} and reasoning as in Lemma \ref{l:slicing2}, we have 
	\begin{align}
	\mathcal G\big(r^\alpha u(r,\theta)\big)&\le  \int_0^1 \mathcal F\big(u(r,\cdot)\big)\, r^{2\alpha+d-3}\, dr+C_{\text{\tiny\sc sl}}\! \int_{0}^1\int_{\de B_1} |\partial_r u|^2\,r^{2\alpha+d-1}\,d\HH^{d-1}dr\notag\\
	&=  \int_0^{1}\mathcal F\big(\varphi(-\kappa \ln r)\big)\, r^{2\alpha+d-3}\, dr+C_{\text{\tiny\sc sl}}\! \int_{0}^1\kappa^2 r^{2\alpha+d-3}\|\varphi'(-\kappa \ln r)\|_{L^2(\partial B_1)}^2\,dr\notag\\
	&= \frac1\kappa\int_0^{+\infty}\mathcal F\big(\varphi(t)\big)\, e^{-\frac{(2\alpha+d-2)t}{\kappa}}\, dt+\kappa\, C_{\text{\tiny\sc sl}}\!\int_{0}^{+\infty} \|\varphi'(t)\|_{L^2(\partial B_1)}^2 e^{-\frac{(2\alpha+d-2)t}{\kappa}}\,dt.\label{e:G1}
	\end{align}
	In particular, this implies that 
	\begin{align}
	\mathcal G\big(r^\alpha u(r,\theta)\big)-\mathcal G\big( r^\alpha c(\theta)\big)&\le \frac1\kappa\int_0^T \Big(\mathcal F\big(\psi(t)\big) -\mathcal F\big(\psi(0)\big)\Big)\, e^{-\frac{(2\alpha+d-2)t}{\kappa}}\,dt\notag\\
	&\qquad\qquad +\frac1{d+2\alpha-2} \Big(\mathcal F\big(\psi(T)\big) -\mathcal F\big(\psi(0)\big)\Big)\,  e^{-\frac{(2\alpha+d-2)t}{\kappa}}\notag\\
	&\qquad\qquad \qquad\qquad+\kappa\, C_{\text{\tiny\sc sl}}\!\int_{0}^T\|\psi'(t)\|_{L^2(\partial B_1)}^2   \, e^{-\frac{(2\alpha+d-2)t}{\kappa}}\,dt.\label{e:G2}
	\end{align}
	Moreover, integrating by parts the first term on the right-hand side, we get 
	\begin{align}
	\mathcal G\big(r^\alpha u(r,\theta)\big)-\mathcal G\big( r^\alpha c(\theta)\big)&\le \frac1{d+2\alpha-2}\int_0^T \psi'(t)\cdot \nabla\mathcal F\big(\psi(t)\big)\,e^{-\frac{(2\alpha+d-2)t}{\kappa}}\,dt\notag\\
	&\qquad\qquad \qquad\qquad+\kappa\, C_{\text{\tiny\sc sl}}\!\int_{0}^T\|\psi'(t)\|_{L^2(\partial B_1)}^2   \, e^{-\frac{(2\alpha+d-2)t}{\kappa}}\,dt.\label{e:G3}
	\end{align}
	Recall that by \eqref{e:decreasing_energy}, the energy is decreasing along the flow. Thus, the first term in the right-hand side of \eqref{e:G2} is negative (thus we simply estimate it from above by zero). Finally, multiplying \eqref{e:G2} and \eqref{e:G3} by $1/2$ and summing them, we obtain the following estimate
	\begin{align}
	\mathcal G\big(r^\alpha u(r,\theta)\big)-\mathcal G\big( r^\alpha c(\theta)\big)&\le \frac{1}{2(d+2\alpha-2)} e^{-\frac{(2\alpha+d-2)T}{\kappa}}\Big(\mathcal F\big(\psi(T)\big) -\mathcal F\big(\psi(0)\big)\Big)\notag\\
	&\qquad\qquad +\frac1{2(d+2\alpha-2)}\int_0^T \psi'(t)\cdot \nabla\mathcal F\big(\psi(t)\big)\,e^{-\frac{(2\alpha+d-2)t}{\kappa}}\,dt\notag\\
	&\qquad\qquad \qquad\qquad+\kappa\, C_{\text{\tiny\sc sl}}\!\int_{0}^T\|\psi'(t)\|_{L^2(\partial B_1)}^2   \, e^{-\frac{(2\alpha+d-2)t}{\kappa}}\,dt.\label{e:G4}
	\end{align}
	We notice that, up to this point, we used only \eqref{e:slicingG} and \eqref{e:slicingG2}, and an integration by parts. 
	\medskip
	
	\noindent {\bf Estimating the error term.} We now estimate the last term in the right-hand side of \eqref{e:G4}, which is also the only positive one. We notice that the energy dissipation condition \eqref{e:energy_dissipation_condition} is equivalent to the following :
	\begin{equation}\label{e:energy_dissipation_condition2}
	\|\psi'(t)\|_{L^2(\partial B_1)}^2\le C_1\Big[-\psi'(t)\cdot\nabla \mathcal F\big(\psi(t)\big)\Big]+C_2\Big[-\psi'(t)\cdot\nabla \mathcal F\big(\psi(t)\big)\Big]^{\sfrac2p},
	\end{equation}
	for some positive constant $C_1$ and $C_2$. As a consequence, we can estimate
	
	\begin{align*}
	\kappa\,C_{\text{\sc sl}}\!
	\int_{0}^T\|\psi'(t)\|_{L^2(\partial B_1)}^2   \, e^{-\frac{(2\alpha+d-2)t}{\kappa}}\,dt&\le \kappa\,C_{\text{\sc sl}}\,C_1\!\int_{0}^T\Big(-\psi'(t)\cdot\nabla \mathcal F\big(\psi(t)\big)\Big)  \, e^{-\frac{(2\alpha+d-2)t}{\kappa}}\,dt\\  
	&\qquad +\kappa\,C_{\text{\sc sl}}\,C_2\!\int_{0}^T\Big(-\psi'(t)\cdot\nabla \mathcal F\big(\psi(t)\big)\Big)^{\sfrac2p}  \, e^{-\frac{(2\alpha+d-2)t}{\kappa}}\,dt.  
	\end{align*}
	Now, the first term on the right hand side can be absorbed into the first term of \eqref{e:G4} by choosing $\kappa$ small enough, in function of the constants involved. In order to estimate the second term, we use the H\"older inequality :
	\begin{align*}
	\int_{0}^T\Big(-\psi'(t)\cdot&\nabla \mathcal F\big(\psi(t)\big)\Big)^{\sfrac2p}  \, e^{-\frac{(2\alpha+d-2)t}{\kappa}}\,dt\\
	&\le  \left(\int_{0}^T-\psi'(t)\cdot\nabla \mathcal F\big(\psi(t)\big)\, e^{-\frac{(2\alpha+d-2)t}{\kappa}}\,dt\right)^{\sfrac2p} \left(\int_{0}^T e^{-\frac{(2\alpha+d-2)t}{\kappa}}\,dt\right)^{1-\sfrac2p}\\
	&\le  \left(\int_{0}^T-\psi'(t)\cdot\nabla \mathcal F\big(\psi(t)\big)\, e^{-\frac{(2\alpha+d-2)t}{\kappa}}\,dt\right)^{\sfrac2p} \left(\frac{\kappa}{2\alpha+d-2}\left(1- e^{-\frac{(2\alpha+d-2)T}{\kappa}}\right)\right)^{1-\sfrac2p}\\
	&\le   \left(\int_{0}^T-\psi'(t)\cdot\nabla \mathcal F\big(\psi(t)\big)\, e^{-\frac{(2\alpha+d-2)t}{\kappa}}\,dt\right)^{\sfrac2p} \left(\frac{\kappa}{2\alpha+d-2}\right)^{1-\sfrac2p}.
	\end{align*} 
	In conclusion, we obtain 
	\begin{align}
	\!\!\! \kappa\,C_{\text{\sc sl}}\!
	\int_{0}^T\|\psi'(t)\|_{L^2(\partial B_1)}^2   \, e^{-\frac{(2\alpha+d-2)t}{\kappa}}\,dt&\le C\kappa^{2-\sfrac2p} \left(\int_{0}^T-\psi'(t)\cdot\nabla \mathcal F\big(\psi(t)\big)\, e^{-\frac{(2\alpha+d-2)t}{\kappa}}\,dt\right)^{\sfrac2p}\notag\\
	&\qquad\qquad+ C\,\kappa  \int_{0}^T-\psi'(t)\cdot\nabla \mathcal F\big(\psi(t)\big)\, e^{-\frac{(2\alpha+d-2)t}{\kappa}}\,dt\,,\label{e:error_estimate_G}
	\end{align}
	where $C$ is a constant depending on $d,\alpha,p, C_{\text{\sc sl}}$ and $C_{\text{\sc ed}}$. 
	\medskip
	
	\noindent {\bf Stopping time.} Recall that, by hypothesis, the flow $\psi$ is defined on the interval $[0,T_{\text{max}}]$.
	We define $T_{\sfrac12}$ as
	$$T_{\sfrac12}=\sup\Big\{s\in[0,T_{\text{max}}]\ :\ \mathcal F(\psi(t))-\mathcal F(\mathcal S)\ge \frac12\big(\mathcal F(\psi(0))-\mathcal F(\mathcal S)\big)\quad\text{for every}\quad t\in[0,s]\Big\},$$
	and we consider two cases. Below, we will choose the stopping time $T$ such that 
	$$0\le T\le T_{\sfrac12}.$$
	\medskip
	
	\noindent{\bf Choice of $\kappa$.} We choose 
	\begin{equation}\label{e:choice_of_kappa}
	\kappa =\eps_\kappa \left(\int_{0}^T-\psi'(t)\cdot\nabla \mathcal F\big(\psi(t)\big)\, e^{-\frac{(2\alpha+d-2)t}{\kappa}}\,dt\right)^{\frac{p-2}{2p-2}},
	\end{equation}
	where $\eps_\kappa>0$ is a small constant, depending on $d,\alpha,p, C_{\text{\sc sl}}$ and $C_{\text{\sc ed}}$, such that 
	\begin{equation}\label{e:choice_of_eps_kappa}
	\eps_\kappa\le 1\ ,\qquad\eps_\kappa C\le \frac1{10}\,\frac{1}{2(2\alpha+d-2)}\qquad\text{and}\qquad \eps_k\le T_{\text{max}}\,.
	\end{equation}
	We notice that by the choice $T\le T_{\sfrac12}$, we have that:
	\begin{align*}
	\int_{0}^T-\psi'(t)\cdot\nabla \mathcal F\big(\psi(t)\big)\, e^{-\frac{(2\alpha+d-2)t}{\kappa}}\,dt&\le \int_{0}^T-\psi'(t)\cdot\nabla \mathcal F\big(\psi(t)\big)\,dt\\
	&=\mathcal F\big(\psi(0)\big)-\mathcal F\big(\psi(T)\big)\le \mathcal F\big(\psi(0)\big)-\mathcal F(\mathcal S)\le E, 
	\end{align*} 
	which gives that 
	$$\kappa\le \eps_\kappa E^{\frac{p-2}{2p-2}}\le \eps_\kappa,$$
	where the last inequality holds when $E\le 1$.
	
	Now, notice that the last term of the right-hand side of \eqref{e:G4} can be estimated as follows :
	\begin{align}
	\kappa\,C_{\text{\sc sl}}\!
	\int_{0}^T\|\psi'(t)\|_{L^2(\partial B_1)}^2   \, e^{-\frac{(2\alpha+d-2)t}{\kappa}}\,dt&\le 2C\eps_k\int_{0}^T-\psi'(t)\cdot\nabla \mathcal F\big(\psi(t)\big)\, e^{-\frac{(2\alpha+d-2)t}{\kappa}}\,dt\notag\\
	&\le \frac1{4(2\alpha+d-2)}\int_{0}^T-\psi'(t)\cdot\nabla \mathcal F\big(\psi(t)\big)\, e^{-\frac{(2\alpha+d-2)t}{\kappa}}\,dt,\label{e:G41/2} 
	\end{align}
	where the first inequality follows by the first inequality for $\eps_k$ in \eqref{e:choice_of_eps_kappa} and the second one is a consequence of the second bound for $\eps_k$ in \eqref{e:choice_of_eps_kappa}. This is the estimate in which we use the first two inequalities in the choice of the constant $\eps_\kappa$. The last inequality of \eqref{e:choice_of_eps_kappa} is only needed for the bound
	$$\kappa\le T_{\text{max}},$$
	which we will use in the two possible choices of $T$ that we discuss below. Before we proceed with the choice of $T$, we notice that  by combining the inequalities \eqref{e:G41/2} and \eqref{e:G4}, we can eliminate the last term in the right-hand side of \eqref{e:G4}. Precisely, the energy gap $\mathcal G(r^\alpha u)-\mathcal G(z)$ can be estimated as follows: 
	
	\begin{align}
	\mathcal G\big(r^\alpha u(r,\theta)\big)-\mathcal G\big( r^\alpha c(\theta)\big)&\le \frac{1}{2(d+2\alpha-2)} e^{-\frac{(2\alpha+d-2)T}{\kappa}}\Big(\mathcal F\big(\psi(T)\big) -\mathcal F\big(\psi(0)\big)\Big)\notag\\
	&\qquad\qquad +\frac1{4(d+2\alpha-2)}\int_0^T \psi'(t)\cdot \nabla\mathcal F\big(\psi(t)\big)\,e^{-\frac{(2\alpha+d-2)t}{\kappa}}\,dt.\label{e:G5}
	\end{align}

	\noindent {\bf Choice of the stopping time.} We now proceed with the choice of $T$, which is the last point of the construction of the competitor. As in Subsection \ref{sub:constructive}, we consider two cases.
	
	\medskip
	
	\noindent {\bf Case 1. \it The energy decreases rapidly along the flow: $T_{\sfrac12}\le \kappa$.} 
	\smallskip
	
	\noindent In this case, we choose $T=T_{\sfrac12}$ and we estimate the first term in the right-hand side of \eqref{e:G5}.
	Indeed, since $\frac{T}{\kappa}\le 1$ and since the function $x\mapsto -e^{-x}$ is increasing in $x$, we have :
	\begin{align*}
	\frac{-e^{-\frac{T(2\alpha+d-2)}\kappa}}{2(2\alpha+d-2)} \big(\mathcal F(\psi(0)) -\mathcal F(\psi(T))\big)&\le \frac{-e^{-(2\alpha+d-2)}}{2(2\alpha+d-2)} \big(\mathcal F(\psi(0)) -\mathcal F(\psi(T))\big)\\
	&=-\frac12 \frac{e^{-(2\alpha+d-2)}}{2(2\alpha+d-2)} \big(\mathcal F(\psi(0))-\mathcal F(\mathcal S)\big),
	\end{align*}
	which concludes the proof of \eqref{e:logepiK} in this case.
\medskip
	
	\noindent {\bf Case 2. \it The energy decreases slowly along the flow: $\kappa\le T_{\sfrac12}$.}
	
	\noindent  In this case, we choose $T=\kappa$ and we estimate the second term in the right-hand side of \eqref{e:G5}.
	By the \L ojasiewicz inequality \eqref{e:lojaK}, we have 
	\begin{align*}
	-\int_0^T -\psi'(t)\cdot\nabla\mathcal F(\psi(t))\,e^{-\frac{t(2\alpha+d-2)}\kappa}\,dt
	&\le -C_{\text{\sc ls}}\int_0^T \big(\mathcal F(\psi(t))-\mathcal F(Q)\big)^{1+\beta}\,e^{-\frac{t(2\alpha+d-2)}\kappa}\,dt\\
	&\le -\frac{C_{\text{\sc ls}}}{2^{1+\gamma}}\int_0^T \big(\mathcal F(\psi(0))-\mathcal F(Q)\big)^{1+\beta}\,e^{-\frac{t(2\alpha+d-2)}\kappa}\,dt\\
	&= -\frac{C_{\text{\sc ls}}\kappa}{(2\alpha+d-2)2^{1+\beta}}\left(1-e^{-\frac{T(2\alpha+d-2)}\kappa}\right) \big(\mathcal F(\psi(0))-\mathcal F(Q)\big)^{1+\gamma}\\
	&= -\frac{C_{\text{\sc ls}}\left(1-e^{-(2\alpha+d-2)}\right) }{(2\alpha+d-2)2^{1+\beta}} \, \kappa\, \big(\mathcal F(\psi(0))-\mathcal F(Q)\big)^{1+\beta},
	\end{align*}
	where the second inequality follows from the fact that 
	$$\mathcal F(\varphi(t))-\mathcal F(Q)\ge \frac12\big(\mathcal F(\varphi(0))-\mathcal F(Q)\big)\quad\text{for every}\quad t\le T=\kappa\le T_{\sfrac12}.$$
	Now, setting 
	$$C=\frac{\eps_k C_{\text{\sc ls}}\left(1-e^{-(2\alpha+d-2)}\right) }{(2\alpha+d-2)2^{1+\beta}}$$
	and using the definition of $\kappa$, we get that 
	\begin{align*}
	-\int_0^T -\psi'(t)&\cdot\nabla\mathcal F(\psi(t))\,e^{-\frac{t(2\alpha+d-2)}\kappa}\,dt\\
	&\le -C\left(\int_{0}^T-\psi'(t)\cdot\nabla \mathcal F\big(\psi(t)\big)\, e^{-\frac{(2\alpha+d-2)t}{\kappa}}\,dt\right)^{\frac{p-2}{2p-2}}\big(\mathcal F(\psi(0))-\mathcal F(Q)\big)^{1+\beta},
	\end{align*} 
	which implies 
	\begin{align*}
	-\left(\int_0^T -\psi'(t)\cdot\nabla\mathcal F(\psi(t))\,e^{-\frac{t(2\alpha+d-2)}\kappa}\,dt\right)^{\frac{p}{2p-2}}\le -C\big(\mathcal F(\psi(0))-\mathcal F(Q)\big)^{1+\beta},
	\end{align*} 
	and finally, 
	\begin{align*}
	-\int_0^T -\psi'(t)\cdot\nabla\mathcal F(\psi(t))\,e^{-\frac{t(2\alpha+d-2)}\kappa}\,dt\le -C^{2-\sfrac2p}\big(\mathcal F(\psi(0))-\mathcal F(Q)\big)^{(1+\beta)\big(2-\frac2p\big)},
	\end{align*}
	which concludes the proof of Theorem \ref{p:epik}, since $\ds1+\gamma=(1+\beta)\big(2-\sfrac2p\big)$. \qed

\section{Rate of convergence of the blow-up sequences}\label{app:B}
In this section, we show how to deduce the rate of convergence of the blow-up sequence starting from  the log-epiperimetric inequality. The argument holds for a general energy $\mathcal E$ and can be used in several different contests: for the obstacle and the thin-obstacle problems, as well as for Bernoulli-type free boundary problems and minimal surfaces (see, for instance \cite{esv1} and \cite{esv2}). 
\begin{prop}\label{p:app:B}
Let $\alpha>0$ be fixed. 
Let the function $u\in H^1(B_1)$ and the energy $\mathcal E:H^1(B_1)\to\R$ be given and,  for every $0<r\le 1$, let $u_r\in H^1(B_1)$ be defined as 
$$u_r(x):=\frac{1}{r^\alpha}u(rx)\quad\text{for every}\quad x\in B_1.$$ 

\begin{enumerate}[\quad \rm (a)]

\item The function $r\mapsto \mathcal E(u_r)$ is differentiable on $(0,1]$ and 
\begin{equation}\label{e:C_a}
\frac{\partial }{\partial r}\mathcal E(u_r)\ge \frac{C_a}{r} \mathcal D(u_r)\qquad\text{for every}\qquad 0<r<1,
\end{equation}
where $C_a>0$ is a given constant and 
$$\mathcal D(u):=\int_{\partial B_1}|x\cdot\nabla u-
\alpha u|^2\,d\HH^{d-1}(x).$$
\item There is a constant $C_b>0$ such that
$$\frac{\partial }{\partial r}\mathcal E(u_r)\ge \frac{C_b}{r}\big(\mathcal E(z_r)-\mathcal E(u_r)\big)\qquad\text{for every}\qquad 0<r<1,$$ 
where $z_r:B_1\to \R$ is the $\alpha$-homogeneous extension of $u_r|_{\partial B_1}$, that is, 
$$z_r(x)=|x|^\alpha u_r\big(\sfrac{x}{|x|}\big)\quad\text{for every}\quad x\in B_1.$$   
\item There are constants $C_c>0$ and $\gamma\in[0,1)$ such that, for every $r\in]0,1]$, there exists a function $h_r\in H^1(B_1)$ for which the following log-epiperimetric inequality holds :
$$\mathcal E(h_r)\le \big(1-C_c|\mathcal E(z_r)|^\gamma\big)\mathcal E(z_r). $$ 
\item For every $0<r\le 1$, we have 
$$0\le \mathcal E(u_r)\le \mathcal E(z_r)\qquad\text{and}\qquad 0\le \mathcal E(u_r)\le \mathcal E(h_r).$$
\end{enumerate}
 
\noindent Then, for every $u\in H^1(B_1)$ satisfying hypotheses {\rm(a)}, {\rm(b)}, {\rm(c)} and {\rm(d)}, and such that $\mathcal E(u)\le E$, for some constant $E$, there exists $u_0\in H^1(B_1)$ such that 
$$\|u_r-u_0\|_{L^2(\partial B_1)}\le C(-\ln r)^{-\frac{1-\gamma}{2\gamma}}\qquad\text{for every}\qquad 0<r\le 1,$$
where the constant $C$ depends on $C_a$, $C_b$, $C_c$, the dimension $d$, the exponent $\gamma$, and on $E$.  
\end{prop}
\begin{proof}
First, notice that by {\rm(b)}, {\rm(c)} and {\rm(d)}, we have 
\begin{align}
\frac{\partial}{\partial r}\mathcal E(u_r)&\ge\frac{C_b}{r}\big(\mathcal E(z_r)-\mathcal E(u_r)\big)\notag\\
&\ge\frac{C_b}{r}\big(\mathcal E(h_r)+C_{c}\,\mathcal E(z_r)^{1+\gamma}-\mathcal E(u_r)\big)\ge \frac{C_bC_c}{r}\mathcal E(u_r)^{1+\gamma}.\label{e:estimate_nome_a_caso}
\end{align}
Consider the change of coordinates $t(r)=-\log r$ (thus, $r(t)=e^{-t}$ and $r'(t)=-r(t)$), and let
$$e(t):=\mathcal E(u_{r(t)})\qquad\text{and} \qquad f(t):=\mathcal D(u_{r(t)}),$$
for every $t\ge 0$. Then, we have  
$$e'(t)=r'(t)\frac{\partial}{\partial r}\mathcal E(u_{r(t)})=-r(t)\frac{\partial}{\partial r}\mathcal E(u_{r(t)}).$$
In particular, using \eqref{e:C_a} and \eqref{e:estimate_nome_a_caso}
$$e'(t)\le -C_a\, f(t)\qquad\text{and}\qquad e'(t)\le -C_b\,C_c\,e(t)^{1+\gamma}.$$
The second inequality implies the decay of $e(t)$. Indeed, 
$$\frac{\partial }{\partial t}\left[e(t)^{-\gamma}-\gamma {t} C_bC_c\right]=\gamma \big(-e(t)^{-1-\gamma}e'(t)-C_bC_c\big)\ge0,$$
which implies that, for every $t\ge 0$, 
$$e(t)^{-\gamma}-\gamma {t} C_bC_c\ge e(0)^{-\gamma},$$
which after rearranging the terms gives
$$e(t)\le \left(e(0)^{-\gamma}+t\gamma C_bC_c\right)^{-\sfrac1\gamma}\quad\text{for every}\quad t\ge 0.$$
In particular, there is a constant $C$, depending on $C_b$, $C_c$, $e(0)$ and $\gamma$, such that 
\begin{equation}\label{e:energy_estimate_t}
e(t)\le C\,t^{-\sfrac1\gamma}\quad\text{for every}\quad t\ge 1.
\end{equation}
Let now $0<r<R\le 1$, $t=-\ln R$ and $T=-\ln r$ be fixed; in particular, $0\le t<T<+\infty$.  
\\
For every $x\in\partial B_1$ we compute  
$$\frac{\partial}{\partial t}u_t(x)=\frac{\partial}{\partial t}\left[\frac{u(tx)}{t^\alpha}\right]=\frac{x\cdot\nabla u(tx)}{t^\alpha}-\frac{\alpha}{t}\frac{u(tx)}{t^\alpha}=\frac1t\big(x\cdot \nabla u_t(x)-\alpha u_t(x)\big).$$
Integrating over $\partial B_1$, we get 
\begin{align*}
\int_{\partial B_1}\left| u_{R}-u_{r}\right|^2 \,d\HH^{d-1}&\leq \int_{\partial B_1}\left(\int_{r}^{R}\frac{1}{\rho}\left|  x\cdot \nabla u_\rho-u_\rho\right| \,d\rho\right)^2\,d\HH^{d-1}\\
&= \int_{\de B_1}\left(\int_{t}^{T}\left|  x\cdot \nabla u_{\rho(\tau)}-u_{\rho(\tau)}\right| \,d\tau\right)^2\,d\HH^{d-1},
\end{align*}
where we used the change of variables $\tau= -\ln \rho$. By the Cauchy-Schwartz inequality, we get 
\begin{align*}
\int_{\partial B_1}\left| u_{R}-u_{r}\right|^2 \,d\HH^{d-1}
&\le  \int_{\partial B_1}\left((T-t)\int_{t}^{T}\left|  x\cdot \nabla u_{\rho(\tau)}-u_{\rho(\tau)}\right|^2 \,d\tau\right)\,d\HH^{d-1}\\
&= (T-t)\int_{t}^{T}\int_{\partial B_1}\left|  x\cdot \nabla u_{\rho(\tau)}-u_{\rho(\tau)}\right|^2d\HH^{d-1}\,d\tau=: (T-t)\int_{t}^{T}f(\tau)\,d\tau\,.
\end{align*}
Now, using the inequality $\ds f(\tau)\le -\frac1{C_a}e'(\tau)$, and integrating in $\tau$, we obtain
$$\int_{\partial B_1}\left| u_{R}-u_{r}\right|^2 \,d\HH^{d-1}\le \frac{T-t}{C_a}\big(e(t)-e(T)\big)\le \frac{T-t}{C_a} e(t).$$
Applying the above inequality to  
$$T=t_{n+1}=2^{n+1}\ ,\quad t=t_n=2^n\ ,\quad r=r_{n+1}=e^{-2^{n+1}}\ ,\quad R=r_n=e^{-2^{n}},$$
and using \eqref{e:energy_estimate_t}, we get  
\begin{align*}
\int_{\partial B_1}\left|u_{r_{n+1}}-u_{r_n}\right|^2 \,d\HH^{d-1}&\le \frac1{C_a}(T-t)e(t)\le \frac{C}{C_a}  \left(2^{\frac{1-\gamma}{\gamma}}\right)^{-n}.
\end{align*}
Let now $\sigma=2^{-\frac{1-\gamma}{2\gamma}}$. Thus, $\sigma<1$ and 
\begin{align*}
\left\|u_{r_{n+1}}-u_{r_n}\right\|_{L^2(\partial B_1)} &\le \left(\sfrac{C}{C_a}\right)^{\sfrac12} \sigma^{n},
\end{align*}
which implies that, for every $N\in\N$ and for every $m>n\ge N$, we have 
\begin{align*}
\left\|u_{r_{m}}-u_{r_n}\right\|_{L^2(\partial B_1)} &\le\frac{\left(\sfrac{C}{C_a}\right)^{\sfrac12}}{1-\sigma}  \sigma^{N},
\end{align*}
which proves that $u_{r_n}$ is a Cauchy sequence in $L^2(\partial B_1)$ and so, it converges to some $u_0\in L^2(\partial B_1)$, for which we have 
$$\left\|u_{r_{n}}-u_{0}\right\|_{L^2(\partial B_1)} \le\frac{\left(\sfrac{C}{C_a}\right)^{\sfrac12}}{1-\sigma}  \sigma^{n}.$$
In order to conclude the proof, it only remains to notice that if $r\in(r_{n+1},r_n)$, then 
$$\int_{\partial B_1}\left| u_{r_n}-u_{r}\right|^2 \,d\HH^{d-1}\le \frac{t_{n+1}-t_n}{C_a} e(t_n)\le \frac{C}{C_a}2^n2^{-\sfrac{n}{\gamma}}=\frac{C}{C_a}\sigma^{2n},$$
which, by the triangular inequality and the fact that $t_n<-\ln r< t_{n+1}$, implies that 
\begin{align*}
\left\|u_{r}-u_{0}\right\|_{L^2(\partial B_1)}&\le \left\|u_{r}-u_{r_n}\right\|_{L^2(\partial B_1)}+\left\|u_{r_{n}}-u_{0}\right\|_{L^2(\partial B_1)}\\
&\le\left(\sfrac{C}{C_a}\right)^{\sfrac12}\Big(1+\frac1{1-\sigma}\Big)\sigma^{n}=\left(\sfrac{C}{C_a}\right)^{\sfrac12}\frac{2-\sigma}{1-\sigma} \left(2^{-n}\right)^{\frac{1-\gamma}{2\gamma}}\\
&=\left[\left(\sfrac{C}{C_a}\right)^{\sfrac12}\frac{2-\sigma}{1-\sigma} 2^{\frac{1-\gamma}{2\gamma}} \right] t_{n+1}^{-\frac{1-\gamma}{2\gamma}}\le\left[\left(\sfrac{C}{C_a}\right)^{\sfrac12}\frac{2-\sigma}{1-\sigma} 2^{\frac{1-\gamma}{2\gamma}}\right] (-\ln r)^{\frac{1-\gamma}{2\gamma}},
\end{align*}
which proves that $u_r$ converges to $u_0$ in $L^2(\partial B_1)$.
\end{proof}

\end{document}